\documentclass[12pt,a4paper,english]{article}

\newcommand{\beq}{\begin{equation}}
\newcommand{\eeq}{\end{equation}}

\usepackage[latin1]{inputenc}
\usepackage{babel}
\usepackage[centertags]{amsmath}
\usepackage{amsfonts}
\usepackage{amssymb}
\usepackage{color}
\usepackage{amsthm}
\usepackage{newlfont}
\usepackage{graphicx}
\usepackage{dsfont}
\usepackage{geometry}
\usepackage{mathrsfs}
\usepackage{authblk}
\usepackage{verbatim}
\usepackage{hyperref}

\newtheorem{theorem}{Theorem}[section]
\newtheorem{lemma}[theorem]{Lemma}
\newtheorem{coroll}[theorem]{Corollary}
\newtheorem{prop}[theorem]{Proposition}
\newtheorem{definition}[theorem]{Definition}
\newtheorem{remark}[theorem]{Remark}

\newcommand{\msc}[1]{\textbf{MSC2010:} #1.}
\newcommand{\keywords}[1]{\textbf{Key words:} #1.}
\newcommand{\ackn}[1]{\textbf{Acknowledgments:} #1.}

\def\RR{\mathbb R}
\def\NN{\mathbb N}
\def\EE{\mathsf E}
\def\PP{\mathsf P}
\def\CC{\mathcal C}
\def\DD{\mathcal D}
\def\eps{\varepsilon}
\def\cF{{\cal F}}
\def\supp{\text{supp}}

\makeatletter  
\@addtoreset{equation}{section}
\def\theequation{\arabic{section}.\arabic{equation}}
\makeatother  
\begin{document}
\title{\textbf{From optimal stopping boundaries \\ to Rost's reversed barriers \\ and the Skorokhod embedding}
}
\author{Tiziano De Angelis\thanks{School of Mathematics, University of Leeds, Woodhouse Lane LS2 9JT, Leeds, UK; \texttt{t.deangelis@leeds.ac.uk}}}
\maketitle
\begin{abstract}
We provide a new probabilistic proof of the connection between Rost's solution of the Skorokhod embedding problem and a suitable family of optimal stopping problems for Brownian motion with finite time-horizon. In particular we use stochastic calculus to show that the time reversal of the optimal stopping sets for such problems forms the so-called Rost's reversed barrier. 
\end{abstract}
\msc{60G40, 60J65, 60J55, 35R35}
\vspace{+8pt}

\noindent\keywords{optimal stopping, Skorokhod embedding, Rost's barriers, free-boundary problems}
\section{Introduction}

In the 60's Skorokhod \cite{Skor65} formulated the following problem: finding a stopping time $\tau$ of a standard Brownian motion $W$ such that $W_\tau$ is distributed according to a given probability law $\mu$. Many solutions to this problem have been found over the past 50 years via a number of different methods bridging analysis and probability (for a survey one may refer for example to \cite{Ob04}). In recent years the study of Skorokhod embedding was boosted by the discovery of its applications to model independent finance and a survey of these results can also be found in \cite{Hob11}.

In this work we focus on the so-called Rost's solution of the embedding (see \cite{Rost71}) and our main contribution is a new fully probabilistic proof of its connection to a problem of optimal stopping. One of the key differences in our approach compared to other existing proofs of this result (\cite{Cox-Wang13b} and \cite{McCon91}) is that we tackle the optimal stopping problem directly. Moreover, we rely only on stochastic calculus rather than using classical PDE methods, as in \cite{McCon91}, or viscosity theory, as in \cite{Cox-Wang13b}. 

Here we consider Rost's solutions expressed in terms of first hitting times of the time-space Brownian motion $(t,W_t)_{t\ge0}$ to a set usually called \emph{reversed barrier} \cite{Chacon85}. 
A purely probabilistic construction of Rost' s barrier relevant to the present work was recently found in \cite{Cox-Pe13} in a very general setting. Cox and Peskir \cite{Cox-Pe13} proved that given a probability measure $\mu$ one can find a unique couple of left continuous functions $b,c:[0,\infty)\to\RR$, with $b$ increasing and $c$ decreasing, such that $W$ stopped at the stopping time $\tau_{b,c}:=\inf\{t>0\,:\,W_t\le c(t)\:\text{or}\: W_t\ge b(t)\}$ is distributed according to $\mu$. The curves $b$ and $c$ are the boundaries of Rost's reversed barrier set and the stopping time $\tau_{b,c}$ fulfils a number of optimality properties, e.g.~it has the smallest truncated expectation among all stopping times realising the same embedding.

The optimal stopping problem object of our study is pointed out in \cite[Remark 17]{Cox-Pe13} and it was originally linked to Rost's embedding via PDE methods by McConnell \cite[Sec.~13]{McCon91}. Let $T>0$, let $\nu$ and $\mu$ be probability measures with cumulative distributions $F_\nu$ and $F_\mu$, denote $B$ a Brownian motion and consider the optimal stopping problem
\begin{align}\label{formula}
\sup_{0\le \tau\le T}\EE G(B_\tau)\quad\text{with}\quad G(x):= 2\int_0^{x}\big(F_\nu(z)-F_\mu(z)\big)dz,\quad x\in\RR
\end{align}
where $\tau$ is a stopping time of $B$. In this paper we prove that under mild assumptions on $\mu$ and $\nu$ (cf.~Section \ref{sec:sett}) it is optimal in \eqref{formula} to stop $(t,B_t)_{t\ge0}$ at the first exit time from an open set $\CC_T\subset [0,T]\times\RR$ (continuation set) which is bounded from above and from below by two right-continuous, monotone functions of time (one of these could be infinite). For each $T>0$ we denote $\DD_T:=\big\{[0,T]\times\RR\big\}\setminus\CC_T$ (stopping set) and we construct a set $\DD_\infty^-$ as the extension to $[0,\infty)$ of the time reversal of the family $\{\DD_T,\,T>0\}$. Then we show that such $\DD^-_\infty$ is a Rost's barrier in the sense that if $W^\nu$ is another Brownian motion (independent of $B$) with initial distribution $\nu$, the first hitting time $\sigma_*$ of $(t,W^\nu_t)$ to the set $\DD^-_\infty$ gives $W^\nu_{\sigma_*}\sim\mu$.

Our study was inspired by the work of McConnell \cite{McCon91}.
He studied a free-boundary problem, motivated by a version of the two sided Stefan problem, where certain boundary conditions were given in a generalised sense that involved the measures $\mu$ and $\nu$ used in \eqref{formula}. His results of existence uniqueness and regularity of the solution relied mostly upon PDE methods and some arguments from the theory of Markov processes. McConnell showed that the free-boundaries of his problem are the boundaries of a Rost's reversed barrier embedding the law $\mu$ (analogously to the curves $b$ and $c$ of \cite{Cox-Pe13}) and he provided some insights as to how these free-boundaries should also be optimal stopping boundaries for problem \eqref{formula}.

In the present paper we adopt a different point of view and begin by performing a probabilistic analysis of the optimal stopping problem \eqref{formula}. We characterise its optimal stopping boundaries and carry out a deep study of the regularity of its value function. It is important to notice that the second derivative of $G$ in \eqref{formula} only exists in the sense of measures (except under the restrictive assumption of $\mu$ and $\nu$ absolutely continuous with respect to the Lebesgue measure) and therefore our study of the optimal stopping problem naturally involves fine properties of Brownian motion's local time (via the occupation time formula). This feature seems {fairly new in the existing literature on finite time-horizon optimal stopping problems} and requires some new arguments for the study of \eqref{formula}. Our analysis of the regularity of the value function $V$ of \eqref{formula} shows that its time derivative $V_t$ is continuous on $[0,T)\times\RR$ (see Proposition \ref{prop:cont-Vt}) although its space derivative $V_x$ may not be. The proof of the continuity of $V_t$ is entirely probabilistic and to the best of our knowledge it represents a novelty in this literature and it is a result is of independent interest. 

Building on the results concerning problem \eqref{formula} we then provide a simple proof of the connection with Rost's embedding (see proof of Theorem \ref{thm:sk}). We would like to stress that our line of arguments is different to the one in \cite{McCon91} and it is only based on probability and stochastic calculus. Moreover our results extend those of \cite{McCon91} relative to the Skorokhod embedding by considering target measures $\mu$ that may have atoms (McConnell instead only looked at continuous measures).

It is remarkable that the connection between problem \eqref{formula} and Rost's embedding hinges on the probabilistic representation of the time derivative of the value function of \eqref{formula} (see Proposition \ref{prop:Ut}). It turns out that $V_t$ can be expressed in terms of the transition density of $(t,B_t)$ killed when leaving the continuation set $\CC_T$; then symmetry properties of the heat kernel allow us to rewrite $V_t$ as the transition density of $(t,W^\nu_t)$ killed when hitting the Rost's reversed barrier $\DD^-_\infty$ (see Lemma \ref{lemma:Hunt}. McConnell obtained the same result via potential theoretic and PDE arguments). The latter result and It\^o's formula are then used to complete the connection in Theorem \ref{thm:sk}.

One should notice that probabilistic connections between optimal stopping and Skorokhod embedding are not new in the literature and there are examples relative for instance to the Az\'ema-Yor's embedding \cite{AY79} (see \cite{Hob07}, \cite{Meil03}, \cite{Ob07} and \cite{Pe99} among others) and to the Vallois' embedding \cite{Val83} (see \cite{CHO08}). For recent developments of connections between control theory, transport theory and Skorokhod embedding one may refer to \cite{BCH16} and \cite{GH-LT14} among others. Our work instead is more closely related to the work of Cox and Wang \cite{Cox-Wang13b} (see also \cite{Cox-Wang13}) where they show that starting from the Rost's solution of the Skorokhod embedding one can provide the value function of an optimal stopping problem whose optimal stopping time is the hitting time of the Rost's barrier. Their result holds for martingales under suitable assumptions and clearly the optimal stopping problem that they find reduces to \eqref{formula} in the simpler case of Brownian motion. An important difference between this work and \cite{Cox-Wang13b} is that the latter starts from the Rost's barrier and constructs the optimal stopping problem, here instead we argue reverse. Methodologies are also very different as \cite{Cox-Wang13b} relies upon viscosity theory or weak solutions of variational inequalities. Results in \cite{Cox-Wang13} and \cite{Cox-Wang13b} have been recently expanded in \cite{GaObdR14} where viscosity theory and reflected FBSDEs have been used to establish the equivalence between solutions of certain obstacle problems and Root's (as well as Rost's) solutions of the Skorokhod embedding problem.

Finally we would like to mention that here we address the question posed in \cite[Rem.~4.4]{Cox-Wang13} of finding a probabilistic explanation for the correspondence between hitting times of Rost's barriers\footnote{To be precise the question in \cite{Cox-Wang13} was posed for Root's barrier (see \cite{Root}), but Root's and Rost's solutions are known to be closely related.} and suitable optimal stopping times. 

When this work was being completed we have learned of a work by Cox, Ob\l\'oj and Touzi \cite{COT15} where optimal stopping and a time reversal technique are also used to construct Root's barriers for the Skorokhod embedding problem with multiple marginals. In the latter paper the authors study directly an optimal stopping problem associated by \cite{Cox-Wang13} to Root's embedding. They prove that the corresponding stopping set is indeed the Root barrier for a suitable target law $\mu$ and, using an iterative scheme, they extend the result to embeddings with multiple marginals. This is done via a sequence of optimal stopping problems nested into one another. 
The approach in \cite{COT15} is probabilistic but the methods are different to the ones described here. Our results rely on $C^1$ regularity properties of the value function for \eqref{formula} whereas, in \cite{COT15}, only continuity of the value function is obtained. The connection between optimal stopping and Root's embedding found in \cite{COT15} uses an approximation scheme starting from finitely supported measures and it holds for target measures $\mu$ which are centered and with finite first moment. The latter assumptions are not needed here and we deal directly with a general $\mu$ without relying on approximations. Root and Rost embedding are somehow the time-reversal of one another and therefore our work and \cite{COT15} nicely complement each other. Although it should be possible to extend our results and methods to a multi-marginal case, this is not a trivial task and is left for future research.

The present paper is organised as follows. In Section \ref{sec:sett} we provide the setting and give the main results. In Section \ref{sec:optSt} we completely analyse the optimal stopping problem \eqref{formula} and its value function whereas Section \ref{sec:sk} is finally devoted to the proof of the link to Rost's embedding. A technical appendix collects some results and concludes the paper.

\section{Setting and main results}\label{sec:sett}

Let $(\Omega,\cF,\PP)$ be a probability space, $B:=(B_t)_{t\ge0}$ a one dimensional standard Brownian motion 
and denote $(\cF_t)_{t\ge0}$ the natural filtration of $B$ augmented with $\PP$-null sets. Throughout the paper we will equivalently use the notations $\EE f(B^x_t)$ and $\EE_x f(B_t)$, for $f:\RR\to \RR$ Borel-measurable, to refer to expectations under the initial condition $B_0=x$.

Let $\mu$ and $\nu$ be probability measures on $\RR$ with $\mu(\{\pm\infty\})=\nu(\{\pm\infty\})=0$, i.e.~with no atoms at infinity. We denote by $F_\mu(x):=\mu((-\infty,x])$ and $F_\nu(x):=\nu((-\infty,x])$ the (right-continuous) cumulative distributions functions of $\mu$ and $\nu$. Throughout the paper we will use the following notation:
\begin{align}
\label{not:a}&a_+:=\sup\{x\in\RR\,:\,x\in\supp\,\nu\}\quad\text{and}\quad a_-:=-\inf\{x\in\RR\,:\,x\in\supp\,\nu\}\\[+4pt]
\label{not:mu}&\mu_+:=\sup\{x\in\RR\,:\,x\in\supp\,\mu\}\quad\text{and}\quad\mu_-:=-\inf\{x\in\RR\,:\,x\in\supp\,\mu\}
\end{align}
and for the sake of simplicity but with no loss of generality we will assume $a_\pm\ge 0$. We also make the following assumptions which are standard in the context of Rost's solutions to the Skorokhod embedding problem (see for example \cite{Cox-Pe13}, and in particular Remark 2 on page 12 therein).
\begin{itemize}
\item[(D.1)] There exist numbers $\hat{b}_+\ge a_+$ and $\hat{b}_-\ge a_-$ such that $(-\hat{b}_-,\hat{b}_+)$ is the largest interval containing $(-a_-,a_+)$ with $\mu((-\hat{b}_-,\hat{b}_+))=0$;
\item[(D.2)] If $\hat b_+=a_+$ (resp.~$\hat b_-=a_-$) then $\mu(\{\hat{b}_+\})=0$ (resp.~$\mu(\{-\hat{b}_-\})=0$). 
\end{itemize}
It should be noted in particular that in the canonical example of $\nu(dx)=\delta_0(x)dx$ we have $a_+=a_-=0$ and the above conditions hold for any $\mu$ such that $\mu(\{0\})=0$.

Assumption (D.2) is made in order to avoid solutions of the Skorokhod embedding problem involving randomised stopping times. On the other hand Assumption (D.1) guarantees that for any $T>0$ the continuation set of problem \eqref{formula} is connected (see also the rigorous formulation \eqref{eq:V} below). Although (D.1) is not necessary for our main results to hold, the study of general non-connected continuation sets would require a case-by-case analysis. The latter would not affect the key principles presented in this work but it substantially increases the difficulty of exposition. In Remark \ref{rem:ncvx} below we provide an example of $\nu$ and $\mu$ which do not meet condition (D.1) but for which our method works in the same way.

The target measure $\mu$ could be entirely supported only on the positive or on the negative real half-line, i.e.~$\supp\{\mu\}\cap\RR_-=\emptyset$ or $\supp\{\mu\}\cap\RR_+=\emptyset$, respectively. In the former case $\hat{b}_-=+\infty$ and $\mu_-=-\hat{b}_+$, whereas in the latter $\hat{b}_+=+\infty$ and $\mu_+=-\hat{b}_-$. For the sake of generality in most of our proofs we will develop explicit arguments for the case of $\mu$ supported on portions of both positive and negative real axis and will explain how these carry over to the other simpler cases as needed.
\vspace{+5pt}

For $0<T<+\infty$ and $(t,x)\in[0,T]\times\RR$ we denote
\begin{align}\label{eq:G}
G(x):=&2\int_0^x\big(F_\nu(z)-F_\mu(z)\big)dz
\end{align}
and introduce the following optimal stopping problem
\begin{align}
\label{eq:V}V(t,x):=&\sup_{0\le \tau\le T-t}\EE_x G(B_\tau)
\end{align}
where the supremum is taken over all $(\cF_t)$-stopping times in $[0,T-t]$. As usual the continuation set $\CC_T$ and the stopping set $\DD_T$ of \eqref{eq:V} are given by
\begin{align}
\label{eq:CC}&\CC_T:=\{(t,x)\in[0,T]\times\RR\,:\,V(t,x)>G(x)\}\\[+3pt]
\label{eq:DD}&\DD_T:=\{(t,x)\in[0,T]\times\RR\,:\,V(t,x)=G(x)\}.
\end{align}
Moreover for $(t,x)\in[0,T]\times\RR$ the natural candidate to be an optimal stopping time is
\begin{align}\label{eq:optst}
\tau_*(t,x)=\inf\big\{s\ge 0\,:\,(t+s,x+B_s)\in \DD_T\big\}\wedge(T-t).
\end{align}
 
Throughout the paper we will often use the following notation: for a set $A\subset [0,T]\times \RR$ we denote $A\cap\{t<T\}:=\{(t,x)\in A:t<T\}$. Moreover we say that a function $t\mapsto f(t)$ is \emph{decreasing} if $f(t+\eps)\le f(t)$ for all $\eps>0$ and \emph{striclty decreasing} if the inequality is strict. Finally, we use $f(t+)$ and $f(t-)$ to denote the right and left limit, respectively, of $f$ at $t$.

The first result of the paper concerns the geometric characterisation of $\CC_T$ and $\DD_T$ and confirms that \eqref{eq:optst} is indeed optimal for problem \eqref{eq:V}.
\begin{theorem}\label{thm:OS}
The minimal optimal stopping time for \eqref{eq:V} is given by $\tau_*$ in \eqref{eq:optst}. Moreover, there exist two right-continuous, decreasing functions $b_+,b_-:[0,T]\to\RR_+\cup\{+\infty\}$, with 
$b_\pm(T-)=\hat{b}_\pm$, such that 
\begin{align}
\label{eq:CC-2}&\CC_T=\big\{(t,x)\in[0,T)\times\RR \,:\,x\in \big(-b_-(t),b_+(t)\big)\big\}\,,\\[+3pt]
\label{eq:DD-2}&\DD_T=\{[0,T]\times\RR\}\setminus \CC_T \,.
\end{align}
\end{theorem}
Theorem \ref{thm:OS} will be proven in Section \ref{sec:optSt}, where a deeper analysis of the boundaries' regularity will be carried out. A number of fundamental regularity results for the value function $V$ will also be provided (in particular continuity of $V_t$ in $[0,T)\times\RR$) and these constitute the key ingredients needed to show the connection to Rost's barrier and Skorokhod embedding. In order to present such connection we must introduce some notation.
\vspace{+5pt}

By arbitrariness of $T>0$, problem \eqref{eq:V} may be solved for any time horizon. Hence for each $T$ we obtain a characterisation of the corresponding value function, denoted now $V^T$, and of the related optimal boundaries, denoted now $b^T_\pm$. It is straightforward to observe that for $T_2>T_1$ one has $V^{T_2}(t+T_2-T_1,x)=V^{T_1}(t,x)$ for all $(t,x)\in[0,T_1]\times\RR$ and therefore, thanks to Theorem \ref{thm:OS}, $b^{T_2}_{\pm}(t+T_2-T_1)=b^{T_1}_\pm(t)$ for $t\in[0,T_1)$ since $G$ is independent of time. We can now consider a time reversed version of our continuation set \eqref{eq:CC-2} and extend it to the time interval $[0,\infty)$.  In order to do so we set $T_0=0$, $T_n=n$, $n\ge 1$, $n\in\NN$ and denote $s^{n}_\pm(t):=b^{T_n}_\pm(T_n-t)$ for $t\in(0,T_n]$. Note that, as already observed, for $m>n$ and $t\in(0,T_n]$ it holds $s^{m}_\pm(t)=s^n_\pm(t)$.
\begin{definition}\label{def:spm}
Let $s_\pm:[0,\infty)\to\RR_+\cup\{+\infty\}$ be the left-continuous increasing functions defined by taking $s_\pm(0):=\hat{b}_\pm$ and
\begin{align*}
s_\pm(t):=\sum^\infty_{j=0}s^{j+2}_\pm(t)\mathds{1}_{(T_j,T_{j+1}]}(t),\quad t\in(0,\infty).
\end{align*}
\end{definition}
For any $T>0$ the curves $s_+$ and $-s_-$ restricted to $(0,T]$ constitute the upper and lower boundaries, respectively, of the continuation set $\CC_T$ after a time-reversal.
The next theorem establishes that indeed $s_+$ and $-s_-$ provide the Rost's reversed barrier which embeds $\mu$. Its proof is given in Section \ref{sec:sk}.
\begin{theorem}\label{thm:sk}
Let $W^\nu:=(W^\nu_t)_{t\ge0}$ be a standard Brownian motion with initial distribution $\nu$ and define
\begin{align}\label{eq:sigma*}
\sigma_*:=\inf\big\{t>0\,:\,W^\nu_t\notin \big(-s_-(t),s_+(t)\big)\big\}.
\end{align}
Then
it holds
\begin{align}\label{eq:sk}
\EE f(W^\nu_{\sigma_*})\mathds{1}_{\{\sigma_*<+\infty\}}=\int_\RR f(y)\mu(dy),\quad\text{for all $f\in C_b(\RR)$.}
\end{align}
\end{theorem}
\begin{remark}\label{rem:uniq}
It was shown in \cite[Thm.~10]{Cox-Pe13} that there can only exist one couple of left-continuous increasing functions $s_+$ and $s_-$ such that our Theorem \ref{thm:sk} holds. Therefore our boundaries coincide with those obtained in \cite{Cox-Pe13} via a constructive method. As a consequence $s_+$ and $s_-$ fulfil the optimality properties described by Cox and Peskir in Section 5 of their paper, i.e.,~$\sigma_*$ has minimal truncated expectation amongst all stopping times embedding $\mu$.
\end{remark}
\begin{remark}\label{rem:inteq}
Under the additional assumption that $\mu$ is continuous we were able to prove in \cite{DeA-sk2} that $s_\pm$ uniquely solve a system of coupled integral equations of Volterra type and can therefore be evaluated numerically.
\end{remark}

\section{Solution of the optimal stopping problem}\label{sec:optSt}

In this section we provide a proof of Theorem \ref{thm:OS} and extend the characterisation of the optimal boundaries $b_+$ and $b_-$ in several directions. Here we also provide a thorough analysis of the regularity of $V$ in $[0,T]\times\RR$ and especially across the two boundaries. Such study is instrumental to the proofs of the next section but it contains numerous results on optimal stopping which are of independent interest. 
\vspace{+5pt}

We begin by showing finiteness, continuity and time monotonicity of $V$.
\begin{prop}\label{prop:V}
For all $(t,x)\in[0,T]\times \RR$ it holds $|V(t,x)|<+\infty$. The map $t\mapsto V(t,x)$ is decreasing for all $x\in\RR$ and $V\in C([0,T]\times\RR)$. Moreover $x\mapsto V(t,x)$ is Lipschitz continuous with constant $L_G$ independent of $t$ and $T$.
\end{prop}
\begin{proof}
Finiteness is a simple consequence of sublinear growth of $G$ at infinity and of $T<+\infty$. Since $G$ is independent of time then $t\mapsto V(t,x)$ is decreasing on $[0,T]$ for each $x\in\RR$ by simple comparison. To show that $V\in C([0,T]\times\RR)$ we take $0\le t_1<t_2\le T$ and $x\in\RR$, then
\begin{align*}
0\le& V(t_1,x)-V(t_2,x)\le \sup_{0\le \tau\le T-t_1}\EE_{x} \big[\big(G(B_\tau)-G(B_{T-t_2})\big)\mathds{1}_{\{\tau\ge T-t_2\}}\big]\\
\le& L_G\EE_{x} \big[\sup_{T-t_2\le s\le T-t_1}\big|B_s-B_{T-t_2}\big|\big]\to0\quad\text{as $t_2-t_1\to 0$}\nonumber
\end{align*}
where we have used that $x\mapsto G(x)$ is Lipschitz on $\RR$ with constant $L_G\in(0,4]$ and the limit follows by dominated convergence. Now we take $x,y\in\RR$ and $t\in[0,T]$, then
\begin{align*}
\big|V(t,x)-V(t,y)\big|\le& L_G\EE \big[\sup_{0\le s\le T-t}\big|B^x_s-B^y_s\big|\big]=  L_G|x-y|.
\end{align*}
Since $V(\,\cdot\,,x)$ is continuous on $[0,T]$ for each $x\in\RR$ and $V(t,\,\cdot\,)$ is continuous on $\RR$ uniformly with respect to $t\in[0,T]$ continuity of $(t,x)\mapsto V(t,x)$ follows.
\end{proof}
\vspace{+5pt}

The above result implies that $\CC_T$ is open and $\DD_T$ is closed (see \eqref{eq:CC} and \eqref{eq:DD}) and standard theory of optimal stopping guarantees that \eqref{eq:optst} is the smallest optimal stopping time for problem \eqref{eq:V}. Moreover from standard arguments, which we collect in Appendix for completeness, $V\in C^{1,2}$ in $\CC_T$ and it solves the following obstacle problem
\begin{align}
\label{eq:pde01}&\big(V_t+\tfrac{1}{2}V_{xx})(t,x)=0, &\text{for $(t,x)\in\CC_T$}\\[+3pt]
\label{eq:pde02}&V(t,x)=G(x), &\text{for $(t,x)\in\DD_T$}\\[+3pt]
\label{eq:pde03}&V(t,x)\ge G(x), &\text{for $(t,x)\in[0,T]\times\RR$.}
\end{align}

We now characterise $\CC_T$ and prove an extended version of Theorem \ref{thm:OS}.
\begin{theorem}\label{thm:OS-b}
All the statements in Theorem \ref{thm:OS} hold and moreover one has
\begin{itemize}
\item[  i)] if $\supp\{\mu\}\subseteq \RR_+$ then $b_-\equiv\infty$ and there exists $t_0\in[0,T)$ such that $b_+(t)<\infty$ for $t\in(t_0,T]$,
\item[ ii)] if $\supp\{\mu\}\subseteq \RR_-$ then $b_+\equiv\infty$ and there exists $t_0\in[0,T)$ such that $b_-(t)<\infty$ for $t\in(t_0,T]$,
\item[iii)] if $\supp\{\mu\}\cap \RR_+\neq\emptyset$ and $\supp\{\mu\}\cap \RR_-\neq\emptyset$ then there exists $t_0\in[0,T)$ such that $b_\pm(t)<\infty$ for $t\in(t_0,T]$,
\item[ iv)] if $\nu(\{a_+\})>0$ (resp.~$\nu(\{-a_-\})>0$) then $b_+(t)>a_+$ for $t\in[0,T)$ (resp.~$b_-(t)>a_-$).
\end{itemize}
Finally, letting $\Delta b_\pm(t):=b_\pm(t)-b_\pm(t-)\le 0$, for any $t\in[0,T]$ such that $b_\pm(t)<+\infty$ it also holds
\begin{align}
\label{eq;jumpb}&\Delta b_+(t)<0\quad\Rightarrow\quad \mu\big(\big(b_+(t),b_+(t-)\big)\big)=0\\
\label{eq;jumpb2}&\Delta b_-(t)<0\quad\Rightarrow\quad \mu\big(\big(-b_-(t-),-b_-(t)\big)\big)=0.
\end{align}
\end{theorem}
\begin{proof}
The proof is provided in a number of steps. 
\vspace{+6pt}

\noindent \emph{Step 1}. Here we prove that $\DD_T\cap\{t<T\}\neq \emptyset$. 

Arguing by contradiction assume that $\DD_T\cap\{t<T\}=\emptyset$. Fix $x\in\supp\,\{\mu\}$ and notice that with no loss of generality we may assume that $\text{dist}(x,\supp\,\nu)\ge 2\eps$ for some $\eps>0$. Indeed if no such $x$ and $\eps$ exist then $(D.1)$ and $(D.2)$ imply $\mu_\pm=\hat{b}_\pm=a_\pm$ with $\mu(\{a_\pm\})=0$, hence a contradiction. 

We define $\tau_{\eps}:=\inf\{t\ge0\,:\, B_t\notin A^x_\eps\}$ with $A^x_\eps:=(x-\eps,x+\eps)$ and also notice that $\mu(A^x_\eps)>0$.
Then for arbitrary $t\in[0,T)$ it holds
\begin{align}\label{eq:OSproof03}
V(t,x)=&G(x)+\int_\RR\EE_{x}L^z_{T-t}(\nu-\mu)(dz)\\
=&G(x)+\int_\RR\EE_{x}L^z_{T-t}\mathds{1}_{\{\tau_\eps\le T-t\}}\nu(dz)-\int_\RR\EE_{x}L^z_{T-t}\mu(dz)\nonumber\\
\le &G(x)+\int_\RR\EE_{x}L^z_{T-t}\mathds{1}_{\{\tau_\eps\le T-t\}}\nu(dz)-\int_{{A^x_\eps}}\EE_{x}L^z_{T-t}\mu(dz)\nonumber
\end{align}
where we have used that $L^z_{T-t}\mathds{1}_{\{\tau_\eps> T-t\}}=0$, $\PP_x$-a.s.~for all $z\in\supp\,\{\nu\}$, since $B_{t\wedge\tau_\eps}\in A^x_\eps$, for all $t\ge0$, $\PP_x$-a.s. We now analyse separately the two integral terms in \eqref{eq:OSproof03}. For the second one we note that
\begin{align}\label{eq:OSproof04}
\int_{{A^x_\eps}}\EE_{x}L^z_{T-t}\mu(dz)=&\int_{{A^x_\eps}}\left(\int_0^{T-t}\frac{1}{\sqrt{2\pi\,s}}e^{-\tfrac{1}{2s}(x-z)^2}ds\right)
\mu(dz)\\
\ge&\mu({A^x_\eps})\int_0^{T-t}\frac{1}{\sqrt{2\pi\,s}}e^{-\tfrac{1}{2s}\eps^2}ds=\mu({A^x_\eps})\EE_0L^\eps_{T-t}\nonumber
\end{align}
where we have used
\begin{align}\label{eq:Eloct}
\EE_{x}L^z_{T-t}=\int_0^{T-t}\frac{1}{\sqrt{2\pi\,s}}e^{-\tfrac{1}{2s}(x-z)^2}ds.
\end{align}
For the first integral in the last line of \eqref{eq:OSproof03} we use strong Markov property and additivity of local time to obtain
\begin{align*}
&\int_\RR\EE_{x}L^z_{T-t}\mathds{1}_{\{\tau_\eps\le T-t\}}\nu(dz)=\int_\RR\EE_{x}\Big[\EE_x\big( L^z_{T-t}\big|\cF_{\tau_\eps}\big)\mathds{1}_{\{\tau_\eps\le T-t\}}\Big]\nu(dz)\\
&=\int_\RR\EE_{x}\Big[\big(\EE_{B_{\tau_\eps}}\big( L^z_{T-t-\tau_\eps}\big)+L^z_{\tau_\eps}\big)\mathds{1}_{\{\tau_\eps\le T-t\}}\Big]\nu(dz)=\int_\RR\EE_{x}\Big[\EE_{B_{\tau_\eps}}\big( L^z_{T-t-\tau_\eps}\big)\mathds{1}_{\{\tau_\eps\le T-t\}}\Big]\nu(dz)\nonumber
\end{align*}
where we have also used $L^z_{\tau_\eps}=0$, $\PP_x$-a.s.~for $z\in\supp\,\{\nu\}$. We denote $A:=\{B_{\tau_\eps}=x+\eps\}$ and $A^c:=\{B_{\tau_\eps}=x-\eps\}$, then given that $t\mapsto L^z_t$ is increasing 
\begin{align*}
&\int_\RR\EE_{x}\Big[\EE_{B_{\tau_\eps}}\big( L^z_{T-t-\tau_\eps}\big)\mathds{1}_{\{\tau_\eps\le T-t\}}\Big]\nu(dz)\le\int_\RR\EE_{x}\Big[\EE_{B_{\tau_\eps}}\big( L^z_{T-t}\big)\mathds{1}_{\{\tau_\eps\le T-t\}}\Big]\nu(dz)\\
&=\int_\RR\Big(\EE_{x+\eps}\big[L^z_{T-t}\big]\EE_{x}\big[\mathds{1}_{\{\tau_\eps\le T-t\}}\mathds{1}_A\big]+\EE_{x-\eps}\big[L^z_{T-t}\big]\EE_{x}\big[\mathds{1}_{\{\tau_\eps\le T-t\}}\mathds{1}_{A^c}\big]\Big)\nu(dz).\nonumber
\end{align*}
Now we recall that $\text{dist}(x,\,\supp\,\nu)\ge 2\eps$ so that by \eqref{eq:Eloct} it follows
\begin{align*}
\EE_{x+\eps}\big[L^z_{T-t}\big]\le \int_0^{T-t}\frac{1}{\sqrt{2\pi\,s}}e^{-\tfrac{1}{2s}\eps^2}ds=\EE_0 L^\eps_{T-t}\quad\text{for all $z\in\supp\,\{\nu\}$}
\end{align*}
and analogously
\begin{align}\label{eq:OSproof08}
\EE_{x-\eps}\big[L^z_{T-t}\big]\le\EE_0 L^\eps_{T-t}\quad\text{ for all $z\in\supp\,\{\nu\}$.}
\end{align}
Adding up \eqref{eq:OSproof04}--\eqref{eq:OSproof08} we find
\begin{align}\label{eq:OSproof09}
V(t,x)\le G(x)+\EE_0(L^\eps_{T-t})\big(\PP_x(\tau_\eps\le T-t)-\mu({A^x_\eps})\big)
\end{align}
and since
\begin{align*}
\lim_{s\downarrow 0}\PP_x(\tau_\eps\le s)=0
\end{align*}
by continuity of Brownian paths, one can find $t$ close enough to $T$ so that $\PP_x(\tau_\eps\le T-t)<\mu({A^x_\eps})$ and \eqref{eq:OSproof09} gives a contradiction. Hence $\DD_T\cap\{t<T\}\neq\emptyset$.
\vspace{+6pt}

\noindent\emph{Step 2}. Here we show that $[0,T)\times(-a_-,a_+)\subseteq\CC_T$ and in particular if $a_-=a_+=0$ then $[0,T)\times\{0\}\subset\CC_T$. Moreover if $\nu(\{\pm a_\pm\})>0$ then also $[0,T)\times\{\pm a_\pm\}\subset\CC_T$, and finally, if $-\hat b_-<\hat b_+$, then $[0,T)\times(-\hat b_-,\hat b_+)\subseteq \CC_T$. We analyse separately the cases in which $\hat b_\pm>a_\pm$ and those in which $\hat b_+=a_+$ and/or $\hat b_-=a_-$.

Assume first $$-\hat b_-<-a_-\le a_+<\hat b_+.$$ Fix $t\in[0,T)$ and $x\in (-\hat{b}_-,\hat{b}_+)$. Under $\PP_x$ we let $\tau_b$ be 
\begin{align*}
\tau_b:=\inf\{s\ge 0\,:\,B_s\notin(-\hat{b}_-,\hat{b}_+)\}\wedge (T-t)
\end{align*}
 and applying It\^o-Tanaka-Meyer's formula we get
\begin{align}\label{eq:OSproof02}
V(t,x)\ge\EE_x G(B_{\tau_b})=&G(x)+\int_\RR \EE_x L^z_{\tau_b}\,(\nu-\mu)(dz) \\
=&G(x)+\int^{a_+}_{-a_-} \EE_x L^z_{\tau_b}\,\nu(dz)>G(x)\nonumber
\end{align}
where $(L^z_t)_{t\ge0}$ is the local time of $B$ at $z\in\RR$.  We have used that $B$ hits any point of $[-a_-,a_+]$ before $\tau_b$ with positive probability under $\PP_x$ whereas $L^z_{\tau_b}=0$, $\PP_x$-a.s.~for all $z\in\supp\,\{\mu\}$. The latter is true because $B_{t\wedge\tau_b}\in(-\hat b_-,\hat b_+)$ for all $t\ge0$, $\PP_x$-a.s. From \eqref{eq:OSproof02} it follows $[0,T)\times(-\hat{b}_-,\hat{b}_+)\subset \CC_T$.

Let us now consider $\hat{b}_+=a_+=0$ and prove that $[0,T)\times\{0\}\subset\CC_T$. From Assumption (D.2) we have~$\mu(\{0\})=0$ and $\nu(\{0\})=1$.
For an arbitrary $\eps>0$ and $t\in[0,T)$ we denote $A_\eps:=(-\eps,+\eps)$ and 
\begin{align*}
\tau_\eps:=\inf\{s\ge 0\,:\,B_s\notin A_\eps\}\wedge(T-t).
\end{align*} 
Then it follows
\begin{align}\label{eq:bound00}
V(t,0)\ge& \EE_{0}G(B_{\tau_\eps})=G(0)+\int_{\RR}\EE_{0}L^z_{\tau_\eps}(\nu-\mu)(dz)\\
=&G(0)+\int_{{A}_\eps}\EE_{0}L^z_{\tau_\eps}(\nu-\mu)(dz).\nonumber
\end{align}
From It\^o-Tanaka's formula we get
\begin{align}
\label{eq:bound01}&\int_{{A}_\eps}\EE_{0}L^z_{\tau_\eps}\nu(dz)= \nu(\{0\})\EE_{0}L^{0}_{\tau_\eps}=\nu(\{0\})\EE_0|B_{\tau_\eps}|\\
\label{eq:bound01.1}&\int_{{A}_\eps}\EE_{0}L^z_{\tau_\eps}\mu(dz)\le\mu({A}_\eps)\EE_0|B_{\tau_\eps}|
\end{align}
where in the last inequality we have used $\EE_{0}L^z_{\tau_\eps}\le\EE_0\big|B_{\tau_\eps}\big|$.
From \eqref{eq:bound00}, \eqref{eq:bound01} and \eqref{eq:bound01.1} we find
\begin{align}\label{eq:bound02}
V(t,0)-G(0)\ge \EE_0|B_{\tau_\eps}|\big(\nu(\{0\})-\mu({A}_\eps)\big)
\end{align}
and for $\eps>0$ sufficiently small the right-hand side of the last equation becomes strictly positive since $\mu({A}_\eps)\to\mu(\{0\})=0$ as $\eps\to0$. 

Notice that the arguments above hold even if $\nu(\{0\})\in(0,1)$, so that the same rationale may be used to show that $\nu(\{\pm a_\pm\})>0\implies [0,T)\times\{\pm a_\pm\}\subset\CC_T$. Hence condition $iv)$ in the statement of the theorem holds as well.

All the remaining cases with $\hat b_+=a_+$ and/or $\hat b_-=a_-$ can be addressed by a combination of the methods above. 
\vspace{+6pt}

\noindent \emph{Step 3}. Here we prove existence and monotonicity of the optimal boundaries. For each $t\in[0,T)$ we denote the $t$-section of $\CC_T$ by
\begin{align}\label{eq:t-secC}
\CC_T(t):=\big\{x\in\RR\,:\,(t,x)\in\CC_T\big\}
\end{align}
and we observe that the family $\big(\CC_T(t)\big)_{t\in[0,T)}$ is decreasing in time since $t\mapsto V(t,x)-G(x)$ is decreasing (Proposition \ref{prop:V}). Next we show that for each $t\in[0,T)$ it holds $\CC_T(t)=(-b_-(t),b_+(t))$ for some $b_\pm(t)\in[a_\pm,\infty]$.

Since $\DD_T\cap\{t<T\}\neq\emptyset$, due to step 1 above, with no loss of generality we assume $x\ge a_+$ and such that $(t,x)\in\DD_T$ for some $t\in[0,T)$ (alternatively we could choose $x\le -a_-$ with obvious changes to the arguments below). It follows that $[t,T]\times\{x\}\in\DD_T$ since $t\mapsto \CC_T(t)$ is decreasing. 

It is sufficient to prove that $(t,y)\in\DD_T$ for $y\ge x$. We argue by contradiction and assume that there exists $y>x$ such that $(t,y)\in \CC_T$. Recall $\tau_*$ in \eqref{eq:optst} and notice that for all $z\in\supp\,\{\nu\}$ we have $L^z_{\tau_*}=0$, $\PP_y$-a.s.~because $\tau_*\le\hat \tau_a$ with $\hat\tau_a$ the first entry time to $[-a_-,a_+]$. Hence we obtain the contradiction:
\begin{align*}
V(t,y)=\EE_y G(B_{\tau_*})=G(y)+\int_\RR\EE_y L^z_{\tau_*}(\nu-\mu)(dz)\le G(y).
\end{align*}

Finally, the maps $t\mapsto b_\pm(t)$ are decreasing by monotonicity of $t\mapsto\CC_T(t)$.
\vspace{+6pt}

\noindent \emph{Step 4}. We now prove conditions $i)$, $ii)$ and $iii)$ on finiteness of the boundaries. In particular we only address $i)$ as the other items follow by similar arguments. 

In step 1 and 3~above we obtained that for any $x\in\supp\,\{\mu\}$, with $x>a_+$, there is $t\in[0,T)$ such that $[t,T]\times [x,+\infty)\subset \DD_T$. Hence the second part of $i)$ follows. 

To prove that $b_-=+\infty$, we recall that $[0,T)\times\supp\{\nu\}\subseteq \CC_T$ from step 2. If $-a_-<a_+$, then for any $x<0$ and $t<T$ a strategy consisting of stopping at the first entry time to $[0,a_+]$, denoted by $\hat{\tau}_0$, gives
\begin{align}\label{eq:arg1}
V(t,x)\ge\EE_x G(B_{\hat{\tau}_0\wedge(T-t)})=G(x)+\int_\RR\EE_xL^z_{\hat{\tau}_0\wedge(T-t)}\nu(dz)>G(x)
\end{align}
because $\supp\,\{\mu\}\subseteq\RR_+$. If instead $a_+=a_-=0$ then there exists $\eps>0$ and $\delta>0$ such that $[0,T-\delta)\times(-\eps,\eps)\subseteq\CC_T$ because $\CC_T$ is open and $(0,T)\times\{0\}\subseteq\CC_T$. Therefore for $x<-\eps$ and $t<T-\delta$ we can repeat the argument used in \eqref{eq:arg1} by replacing $\hat \tau_0$ with 
\begin{align*}
\hat \tau_{\eps,\delta}:=\inf\{s\ge 0\,:\,x+B_s\ge \eps\}\wedge(T-t-\delta). 
\end{align*}
By arbitrariness of $\eps$ and $\delta$ it follows that $[0,T)\times\RR_-\subseteq\CC_T$.
\vspace{+6pt}

\noindent \emph{Step 5}. In this final step we show continuity properties of the boundaries. Right continuity of the boundaries follows by a standard argument which we repeat (only for $b_+$) for the sake of completeness. Fix $t_0\in[0,T)$ and let $(t_n)_{n\in\mathbb{N}}$ be a decreasing sequence such that $t_n\downarrow t_0$ as $n\to\infty$, then $(t_n,b_+(t_n))\to (t_0,b_+(t_0+))$ as $n\to\infty$, where the limit exists since $b_+$ is monotone. Since $(t_n,b_+(t_n))\in\DD_T$ for all $n$ and $\DD_T$ is closed, then it must be $(t_0,b_+(t_0+))\in\DD_T$ and hence $b_+(t_0+)\ge b_+(t_0)$ by definition of $b_+$. Since $b_+$ is decreasing then also $b_+(t_0+)\le b_+(t_0)$ and $b_+$ is right-continuous.

Next we prove \eqref{eq;jumpb}, which is equivalent to say that jumps of $b_\pm$ may only occur if $\mu$ is flat across the jump. For the proof we borrow arguments from \cite{DeA15}. Let us assume that for a given and fixed $t$ we have $b_+(t-)>b_+(t)$ and then take $b_+(t)\le x_1<x_2\le b_+(t-)$ and $0<t'<t$. Notice that the limit $b_+(t-)$ exists because $b_+$ is decreasing. We denote $\cal R$ the rectangular domain with vertices $(t',x_1)$, $(t,x_1)$, $(t,x_2)$, $(t',x_2)$ and denote $\partial_P\cal R$ its parabolic boundary. Then \eqref{eq:pde01} implies that $V\in C^{1,2}({\cal R})$ and it is the unique solution of
\begin{align}\label{eq:cont00}
u_t+\tfrac{1}{2}u_{xx}=0\quad\text{on $\cal R$ with $u=V$ on $\partial_P\cal R$.}
\end{align}
Note that in particular $V(t,x)=G(x)$ for $x\in[x_1,x_2]$. We pick $\psi\in C^\infty_c(x_1,x_2)$ such that $\psi\ge0$ and $\int_{x_1}^{x_2}\psi(y)dy=1$, and multiplying \eqref{eq:cont00} by $\psi$ and integrating by parts we obtain
\begin{align}\label{eq:cont01}
\int^{x_2}_{x_1}V_t(s,y)\psi(y)dy=-\frac{1}{2}\int^{x_2}_{x_1}V(s,y)\psi''(y)dy\quad \text{for $s\in(t',t)$}.
\end{align}
We recall that $V_t\le0$ in $\cal R$ by Proposition \ref{prop:V} and by taking limits as $s\uparrow t$, dominated convergence implies
\begin{align}\label{eq:cont02}
0\le\int^{x_2}_{x_1}V(t,y)\psi''(y)dy=\int^{x_2}_{x_1}G(y)\psi''(y)dy=-2\int^{x_2}_{x_1}\psi(y)\mu(dy)
\end{align}
where we have used that $\nu=0$ on $(x_1,x_2)$ since $b_+(\cdot)\ge a_+$ on $[0,T)$ by step 2 above. Since $(x_1,x_2)$ and $\psi$ are arbitrary we conclude that \eqref{eq:cont02} is only possible if $\mu\big((b_+(t),b_+(t-))\big)=0$.

Finally we prove that $b_\pm(T-)=\hat{b}_\pm$. As usual we only deal with $b_+$ but the same arguments can be used for $b_-$. Recall from step 2~above that $b_+(T-)\ge \hat{b}_+$ and arguing by contradiction we assume that $b_+(T-)>\hat{b}_+$. Then the same steps as in \eqref{eq:cont01}--\eqref{eq:cont02} may be applied to the interval $(\hat b_+,b_+(T-))$, and since $\mu((\hat b_+,b_+(T-)))>0$ by definition of $\hat b_+$ and the fact that $F_\mu$ is right-continuous, then we reach again a contradiction.
\end{proof}
\vspace{+6pt}

The behaviour of $b_{\pm}$ as $t$ approaches $T$ is very important for our purposes and knowing that $b_\pm(T-)=\hat{b}_\pm$ may not be sufficient in some instances. Therefore we provide here a refined result concerning these limits. 
\begin{lemma}\label{lem:flatb}
If $\mu(\{\hat{b}_+\})>0$ (resp.~$\mu(\{-\hat{b}_-\})>0$) then there exists $t_+\in[0,T)$ (resp.~$t_-\in[0,T)$) such that $b_+(t)=\hat{b}_+$ for all $t\in[t_+,T]$ (resp.~$b_-(t)=\hat{b}_-$ for all $t\in[t_-,T]$).
\end{lemma}
\begin{proof}
We give a proof only for $\mu(\{\hat{b}_+\})>0$ as the other case is completely analogous. Here it is convenient to adopt the notation $\EE_{t,x}[\,\cdot\,]=\EE[\,\cdot\,|B_t=x]$ and with no loss of generality to think of $\Omega$ as the canonical space of continuous trajectories so that the shifting operator $\theta_\cdot:\Omega\to\Omega$ is well defined and $\theta_t \{\omega(s)\,,\,s\ge0\}=\{\omega(t+s)\,,\,s\ge0\}$.

Recalling that $\mu(\{\hat{b}_+\})>0\implies \hat{b}_+>a_+$ due to Assumption (D.2) we now argue by contradiction and assume that $[0,T)\times\{\hat{b}_+\}\in\CC_T$. By It\^o-Tanaka-Meyer formula
\begin{align}\label{flat1}
0<V(t,\hat{b}_+)-G(b_+)=\EE_{t,\hat{b}_+}\int_{\RR}L^z_{\tau_*}(\nu-\mu)(dz)\qquad\text{for all $t\in[0,T)$,}
\end{align}
where $\tau_*$ is optimal under $\PP_{t,\hat{b}_+}$, i.e.~$\tau_*:=\inf\{s\ge t\,:\,(s,B_s)\in\DD_T\}\wedge T$ under $\PP_{t,\hat{b}_+}$. We aim now at finding an upper bound for the right-hand side of \eqref{flat1}. Notice that
\begin{align}\label{flat2}
\EE_{t,\hat{b}_+}\int_{\RR}L^z_{\tau_*}(\nu-\mu)(dz)\le -\mu(\{\hat{b}_+\})\EE_{t,\hat{b}_+} L^{\hat{b}^+}_{\tau_*}+\EE_{t,\hat{b}_+}\int_{\RR}L^z_{\tau_*}\nu(dz)
\end{align}
and let us consider the two terms above separately.

For the first term we set $\tau_R:=\inf\{s\ge t\,:\,|B_s-\hat{b}_+|\ge R \}$ under $\PP_{t,\hat{b}_+}$ for some $R>0$ and use that $|B_{\tau_*\wedge\tau_R}-\hat{b}_+|^p\le R^p$ for any $p>0$ to obtain
\begin{align}\label{flat3}
\EE_{t,\hat{b}_+} L^{\hat{b}^+}_{\tau_*}\ge&\, \EE_{t,\hat{b}_+} L^{\hat{b}^+}_{\tau_*\wedge\tau_R}=\EE_{t,\hat{b}_+}|B_{\tau_*\wedge\tau_R}-\hat{b}_+|\nonumber\\
\ge& \frac{1}{R^p}\EE_{t,\hat{b}_+}|B_{\tau_*\wedge\tau_R}-\hat{b}_+|^{1+p}\ge \frac{c_p}{R^p}\EE_{t,\hat{b}_+}\left[(\tau_*\wedge\tau_R-t)^{\frac{1+p}{2}}\right]
\end{align}
where in the last inequality we have used Burkholder-Davis-Gundy inequality and $c_p>0$ is a fixed constant. 

Now for the second term in the right-hand side of \eqref{flat2} we pick $a\in(a_+,\hat{b}_+)$, set $\tau_{a}:=\inf\{s\ge t\,:\,B_s\le a\}$ and use strong Markov property along with the fact that for $z\in\supp\{\nu\}$ it holds $L^z_{s\wedge \tau_a}=0$, $\PP_{t,\hat{b}_+}$-a.s. These give 
\begin{align}\label{flat4}
\hspace{-3pc}
\int_{\RR}\EE_{t,\hat{b}_+}L^z_{\tau_*}\nu(dz)=&\int_{\RR}\EE_{t,\hat{b}_+}\left[\mathds{1}_{\{\tau_*>\tau_a\}}L^z_{\tau_*}\right]\nu(dz)\nonumber\\[+4pt]
=&\int_{\RR}\EE_{t,\hat{b}_+}\left[\mathds{1}_{\{\tau_*>\tau_a\}}\left(L^z_{\tau_a}+\EE_{t,\hat{b}_+}\left[L^z_{\tau_*}\circ\theta_{\tau_a}\big|\cF_{\tau_a}\right]\right)\right]\nu(dz)\\[+4pt]
=&\int_{\RR}\EE_{t,\hat{b}_+}\left[\mathds{1}_{\{\tau_*>\tau_a\}}\EE_{\tau_a,B_{\tau_a}}\left[L^z_{\tau_*}\right]\right]\nu(dz)\nonumber\\[+4pt]
\le &\,\PP_{t,\hat{b}_+}(\tau_*>\tau_a)\int_{\RR}\sup_{t\le s\le T}\EE_{s,a}\left[L^z_{\tau_*}\right]\nu(dz).\nonumber
\end{align}  
Since we are interested in $t\to T$ and $b_+(T-)=\hat b_+$ by Theorem \ref{thm:OS}, with no loss of generality we assume that $b_+(s)\le R$ for $s\in [t,T]$ and for $R>0$ sufficiently large. With no loss of generality then $[a,b_+(s)]\in[\hat b_+-R,\hat b_+ +R]$ for $s\in[t,T]$. The latter implies that $\{\tau_a<\tau_*\}=\{\tau_a<\tau_*\,,\,\tau_a<\tau_R\}$. Therefore, denoting $\delta:=|\hat{b}_+-a|$ we can estimate 
\begin{align}\label{flat5}
\PP_{t,\hat{b}_+}(\tau_a<\tau_*)\le&\, \PP_{t,\hat{b}_+}(\tau_a<\tau_*\wedge\tau_R)\nonumber\\
\le &\,\PP_{t,\hat{b}_+}\left(\sup_{t\le s\le \tau_*\wedge\tau_R}|B_s-\hat{b}_+|\ge \delta\right)\\[+4pt]
\le&\, \frac{1}{\delta^q}\EE_{t,\hat{b}_+}\left[\sup_{t\le s\le \tau_*\wedge\tau_R}|B_s-\hat{b}_+|^q\right]\nonumber\\[+4pt]
\le&\,\frac{c_q}{\delta^q}\EE_{t,\hat{b}_+}\left[|B_{\tau_*\wedge\tau_R}-\hat{b}_+|^q\right]
\le\frac{c'_q}{\delta^q}\EE_{t,\hat{b}_+}\left[(\tau_*\wedge\tau_R-t)^{q/2}\right]\nonumber
\end{align}
where $q>1$ is arbitrary but fixed, and we have used Doob's inequality and Burkholder-Davis-Gundy inequality with $c_q$, $c'_q$ suitable positive constants. 

To simplify notation we set $\mu_0:=\mu(\{\hat{b}_+\})>0$, $C_p:=c_p/R^p$, $C'_q:=c'_q/\delta^q$ and
\begin{align*}
g(t):=\int_{\RR}\sup_{t\le s\le T}\EE_{s,a}\left[L^z_{\tau_*}\right]\nu(dz),
\end{align*}
and observe that $g(t)\downarrow 0$ as $t\to T$ since $(L^z_s)_{t\le s\le T}$ is continuous and $\EE_{T,a}L^z_{T}=0$ for all $z\in\RR$. Plugging estimates \eqref{flat2}--\eqref{flat5} into \eqref{flat1} and choosing $q=1+p$ we obtain
\begin{align}
0<&\,V(t,\hat{b}_+)-G(\hat{b}_+)\nonumber\\
\le&\, -\mu_0 C_p\EE_{t,\hat{b}_+}\left[(\tau_*\wedge\tau_R-t)^{\frac{1+p}{2}}\right]+g(t)C'_q\EE_{t,\hat{b}_+}\left[(\tau_*\wedge\tau_R-t)^{q/2}\right]\nonumber\\
\le&\,\left(g(t)C'_q-\mu_0 C_p\right) \EE_{t,\hat{b}_+}\left[(\tau_*\wedge\tau_R-t)^{q/2}\right].\nonumber
\end{align}
Since $g(t)\downarrow 0$ as $t\to0$, then for $t<T$ but sufficiently close to $T$ we find a contradiction. Therefore there must exist $t_+\in[0,T)$ such that 
$[t_+,T]\times\{\hat{b}_+\}\in\DD_T$ and since $b_+(\,\cdot\,)\ge b_+(T-)= \hat{b}_+$ by Theorem \ref{thm:OS}, then it follows that $b_+(t)=\hat{b}_+$ for all $t\in[t_+,T]$ as claimed.
\end{proof}
\vspace{+6pt}

To link our optimal stopping problem to the study of the Skorokhod embedding it is important to analyse also the case when $T=+\infty$ in \eqref{eq:V} and to characterise the related optimal stopping boundaries. We define
\begin{align}\label{eq:Vinfty}
v(x):=\sup_{\tau\ge0}\EE_x \left[G(B_\tau)\mathds{1}_{\{\tau<+\infty\}}\right],\quad x\in\RR,
\end{align}
and the associated continuation region is
\begin{align}\label{def:CCinfty}
\CC_\infty:=\{x\in\RR\,:\,v(x)>G(x)\}.
\end{align}
For the study of \eqref{eq:Vinfty} it is useful to collect here some geometric properties of $G$. Since $G'(x)=2(F_\nu-F_\mu)(x)$, then the limits
\begin{align*}
G(+\infty):=\lim_{x\to\infty}G(x)\quad\text{and}\quad G(-\infty):=\lim_{x\to-\infty}G(x) 
\end{align*}
exist because $G'$ changes its sign at most once due to (D.1).
Notice however that $G(\pm\infty)$ might be equal to $+\infty$. 

Moreover 
\begin{align}\label{mon1}
\supp\{\mu\}\subseteq \RR_-\implies G'\le 0\:\text{on $\RR$ and}\:\supp\{\mu\}\subseteq \RR_+\implies G'\ge 0\:\text{on $\RR$.}
\end{align}
On the other hand, if $\mu$ is supported on both sides of $[-a_-,a_+]$ (hence $a_\pm<\infty$) then there exists a unique $a_0\in[-a_-,a_+]$ for which $F_\mu\ge F_\nu$ on $(-\infty,a_0)$ and $F_\mu\le F_\nu$ on $(a_0,+\infty)$. Hence $G$ has a unique global minimum at $a_0$. 

The geometric properties of $G$ collected so far and the fact that $G(0)=0$ allow us to conclude that 
\begin{align}\label{eq:supG}
\sup_{x\in\RR}G(x)=\max\{G(+\infty),G(-\infty)\}.
\end{align}
It is worth noticing that atoms of $\mu$ and $\nu$ correspond to discontinuities of $G'$ and if $\mu$ and $\nu$ are purely atomic then $G$ is continuous and piecewise linear.
Finally we have $G$ concave on $(-\infty,-a_-)\cup(a_+,+\infty)$ and convex on $[-a_-,a_+]$ because $G''(dx)=2(\nu-\mu)(dx)$. 

Recalling our notation for $\mu_\pm$ (see \eqref{not:mu}) and using the properties of $G$ illustrated above, we obtain the next characterisation of the value in \eqref{eq:Vinfty} and its continuation region $\CC_\infty$.
\begin{prop}\label{thm:OSinfiniteH}
The value function of \eqref{eq:Vinfty} is given by 
\begin{align*}
v(x)=\max\{G(+\infty),G(-\infty)\},\quad\text{for $x\in\RR$}
\end{align*}
(it could be $v=+\infty$).
Moreover, letting $\CC_\infty$ as in \eqref{def:CCinfty}, the following holds:
\begin{itemize}
\item[  i)] If $\max\{G(+\infty),G(-\infty)\}=+\infty$ then $\CC_\infty=\RR$;
\item[ ii)] If $G(-\infty)<G(+\infty)<+\infty$ then $\CC_\infty=(-\infty,\mu_+)$;
\item[iii)] If $G(+\infty)<G(-\infty)<+\infty$ then $\CC_\infty=(-\mu_-,\infty)$;
\item[ iv)] If $G(+\infty)=G(-\infty)<+\infty$ then $\CC_\infty=(-\mu_-,\mu_+)$.
\end{itemize}
\end{prop}
\begin{proof}
Due to \eqref{eq:supG} we immediately have $v(x)\le \max\{G(+\infty),G(-\infty)\}$ from \eqref{eq:Vinfty}, so we need to prove the reverse inequality.

With no loss of generality we may consider $G(+\infty)=\max\{G(+\infty),G(-\infty)\}$ and regardless of whether or not $G(+\infty)$ is finite we can argue as follows: we pick $\tau_n:=\inf\{t\ge 0:B_t\ge n\}$ so that $v(x)\ge \EE_x G(B_{\tau_n})=G(n)$ because $\PP_x(\tau_n<+\infty)=1$. Taking the limit as $n\to\infty$ we get $v(x)\ge G(+\infty)$ as needed.

If $v=+\infty$ then $\CC_\infty=\RR$ since $G(x)$ is finite for all $x\in\RR$. The geometry of $\CC_\infty$ in the remaining cases can be worked out easily. Let us consider for example the setting of $ii)$. Since $G(+\infty)>G(-\infty)$ then it must be $\supp\{\mu\}\cap\RR_+\neq \emptyset$, due to \eqref{mon1}. It follows that $0\le a_+<\mu_+$, because $\mu_+=a_+$ is ruled out by (D.2). Then $G'> 0$ on $[a_+,\mu_+)$, which implies that $G(x)<G(\mu_+)$ for $x<\mu_+$ and $G(x)=G(\mu_+)$ for all $x\ge\mu_+$. Hence $G(+\infty)=G(\mu_+)$, and since $v(x)=G(+\infty)$ then $\CC_\infty=(-\infty,\mu_+)$. We notice that the argument holds also if $\mu_+=+\infty$. 

The geometry of $\CC_\infty$ in cases $iii)$ and $iv)$ may be obtained by analogous considerations.
\end{proof}

It is useful to remark that if $\CC_\infty=\RR$ then there is no optimal stopping time in \eqref{eq:Vinfty}. 
Now we give a corollary which will be needed in the rest of the paper and follows immediately from the above proposition
\begin{coroll}\label{cor:suppmu}
Let $b^\infty_\pm>0$ (possibly infinite) be such that $-b^\infty_-$ and $b^\infty_+$ are the lower and upper boundary, respectively, of $\CC_\infty$. Then $\supp\{\mu\}\subseteq [-b^\infty_-,b^\infty_+]$ and in particular $b^\infty_+=+\infty$ (resp.~$b^\infty_-=+\infty$) if $\supp\{\mu\}\cap\RR_+=\emptyset$ (resp.~$\supp\{\mu\}\cap\RR_-=\emptyset$).
\end{coroll}

Recall our notation $V^T$ for the value function of problem \eqref{eq:V} with time-horizon $T>0$ and $b^T_\pm$ for the corresponding optimal boundaries. We now characterise the limits of $b^T_\pm$ as $T\to\infty$ and we show that these coincide with $b^\infty_\pm$ of the above corollary as expected.
\begin{prop}\label{prop:s-lim}
Let $b^\infty_\pm$ be as in Corollary \ref{cor:suppmu}, then 
\begin{align*}
\lim_{T\to\infty}b^T_\pm(0)=b^\infty_\pm.
\end{align*}
\end{prop}
\begin{proof}
Note that $(V^T)_{T>0}$ is a family of functions increasing in $T$ and such that $V^T(0,x)\le v(x)$ (cf.~\eqref{eq:Vinfty}). Set
\begin{align}\label{eq:Vinfty01}
V^\infty(x):=\lim_{T\to\infty}V^T(0,x),\quad x\in\RR
\end{align}
and note that $V^\infty\le v$ on $\RR$. To prove the reverse inequality we introduce the stopping times
\begin{align}
\tau_n:=\inf\{t\ge0:B_t\ge n\},\quad \tau_{-m}:=\inf\{t\ge0:B_t\le -m\}
\end{align}
for $n,m\in\NN$. With no loss of generality we consider the case $v(x)=G(+\infty)$ (possibly infinite) as the remaining cases can be dealt with in the same way.
For any $T>0$ and for $x\in(-m,n)$ we have
\begin{align*}
V^T(0,x)\ge \EE_x\left[G(B_{\tau_n\wedge\tau_{-m}\wedge T})\right]
\end{align*}
and since $G$ is bounded on $[-m,n]$ we can take limits as $T\to\infty$ and use dominated convergence to obtain
\begin{align}\label{eq:low0}
V^\infty(x)\ge& \EE_x\left[G(B_{\tau_n\wedge\tau_{-m}})\right]=G(n)\PP_x(\tau_n<\tau_{-m})+G(-m)\PP_x(\tau_n>\tau_{-m})\nonumber\\
=&\frac{x+m}{n+m}G(n)+\frac{n-x}{n+m}G(-m).
\end{align}

The plan now is to take $m\to\infty$ while keeping $n$ fixed. The first term in the last expression above clearly converges to $G(n)$ as $m\to\infty$. For the second term we observe that, since $F_\nu(z)\downarrow 0$ as $z\to-\infty$ and it is monotonic, then there exists $c_n>0$ such that $0\le F_\nu(z)\le n^{-2}$ for $z\in(-\infty,-c_n]$. Hence, taking $m>c_n$ we can estimate
\begin{align}\label{eq:low1}
\frac{1}{m}G(-m)=&\frac{2}{m}\int_{-m}^0(F_\mu-F_\nu)(z)dz\ge -\frac{2}{m}\int_{-m}^0 F_\nu(z)dz\nonumber\\
=&-\frac{2}{m}\left(\int^{-c_n}_{-m}F_\nu(z)dz+\int^0_{-c_n}F_\nu(z)dz\right)\ge-\frac{2}{m}\left(n^{-2}(m-c_n)+c_n\right).
\end{align}
Taking limits as $m\to\infty$ in \eqref{eq:low0} and using \eqref{eq:low1} we obtain
\begin{align*}
V^\infty(x)\ge G(n)-2(n-x) n^{-2}
\end{align*}
and, finally taking $n\to\infty$ we conclude $V^\infty(x)\ge G(+\infty)=v(x)$. Since $x\in\RR$ was arbitrary we have
\begin{align}\label{eq:comp01}
V^\infty(x)=v(x),\quad x\in\RR.
\end{align}

We are now ready to prove convergence of the related optimal boundaries.
Note that if $(0,x)\in\CC_T$ for some $T$, then $v(x)\ge V^S(0,x)\ge V^T(0,x)>G(x)$ for any $S\ge T$, thus implying that the families $(b^T_\pm(0))_{T>0}$ are increasing in $T$ and $(-b^T_-(0),b^T_+(0))\subseteq (-b^\infty_-,b^\infty_+)$ for all $T>0$.
It follows that
\begin{align*}
\tilde{b}_\pm:=\lim_{T\to\infty} b^T_\pm(0) \le b^\infty_\pm.
\end{align*}
To prove the reverse inequality we take an arbitrary $x\in\CC_\infty$ and assume $x\notin(-\tilde{b}_-,\tilde{b}_+)$. Then $v(x)\ge G(x)+\delta$ for some $\delta>0$ and there must exist $T_\delta>0$ such that $V^T(0,x)\ge G(x)+\delta/2$ for all $T\ge T_\delta$ by \eqref{eq:comp01} and \eqref{eq:Vinfty01}. Hence $x\in(-b^T_-(0),b^T_+(0))$ for all $T$ sufficiently large and since $(-b^T_-(0),b^T_+(0))\subseteq(-\tilde{b}_-,\tilde{b}_+)$ we find a contradiction and conclude that $\tilde b_\pm=b^\infty_\pm$.
\end{proof}


\subsection{Further regularity of the value function}\label{sec:C1}

In this section $0<T<+\infty$ is fixed and we use the simpler notation $V=V^T$ unless otherwise specified (as in Corollary \ref{cor:smf1}). We analyse the behaviour of $V_x(t,\,\cdot\,)$ at points $\pm b_\pm(t)$ of the optimal boundaries. We notice in particular that under the generality of our assumptions the map $x\mapsto V_x(t,x)$ may fail to be continuous across $\pm b_\pm(t)$ due to the fact that $G$ is not everywhere differentiable.  

More importantly we prove by purely probabilistic methods that $V_t$ is instead continuous on $[0,T)\times\RR$. This is a result of independent interest which, to the best of our knowledge, is new in the probabilistic literature concerning optimal stopping and free-boundaries. 
For recent PDE results of this kind one may refer instead to \cite{BDM06}.
Some of the proofs are given in appendix since they follow technical arguments which are not needed to understand the main results of the section. We start by providing useful continuity properties of the optimal stopping times.

Thanks to Theorem \ref{thm:OS} we have that the interior of $\DD_T$ is not empty and we denote it by $\DD^\circ_T$. We also introduce the entry time to $\DD^\circ_T$, denoted by 	
\begin{align}\label{def:tau-t}
\tilde \tau_*(t,x):=\inf\big\{s\ge 0\,:\,(t+s,x+B_s)\in \DD^\circ_T\big\}\wedge (T-t).
\end{align}
We recall $\tau_*$ as in \eqref{eq:optst} and notice that 
\begin{align}\label{st-interior}
\tau_*(t,x)=\tilde \tau_*(t,x),~\text{$\PP$-a.s.~for all $(t,x)\in[0,T)\times\RR$}
\end{align}
due to monotonicity of $b_\pm$ and the law of iterated logarithm (this fact is well known  and the interested reader may find a proof for example in \cite[Lemma 6.2]{EJ16} or \cite[Lemma 5.1]{DeAEk16}). 

The next lemma, whose proof is given in appendix for completeness, is an immediate consequence of \eqref{st-interior}. The second lemma below follows from the law of iterated logarithm and its proof is also postponed to the appendix.
\begin{lemma}\label{lem:contST}
Let $(t,x)\in\partial\CC_T$,
then for any sequence $(t_n,x_n)_n\in \CC_T$ such that $(t_n,x_n)\to(t,x)$ as $n\to\infty$ one has
\begin{align}\label{eq:contST00}
\lim_{n\to\infty}\tau_*(t_n,x_n)=0,\quad\PP-a.s.
\end{align}
\end{lemma}
\begin{lemma}\label{cor:contST-t}
Let $(t,x)\in\CC_T$ and assume that $(t_h)_{h\ge0}$ is such that $t_h\uparrow t$ as $h\to\infty$. Then
\begin{align}\label{eq:contST06}
\lim_{h\to\infty}\tau_*(t_h,x)=\tau_*(t,x),\quad\PP-a.s.
\end{align}
and the convergence is monotonic from above.
\end{lemma}
\vspace{+5pt}

A simple observation follows from Proposition \ref{prop:V}, that is
\begin{align}\label{eq:sm-fit0}
\sup_{[0,T]\times\RR}\big|V_x(t,x)\big|\le L_G,
\end{align}
with $L_G$ independent of $T$.
Next we establish refined bounds for $V_x$ at the optimal boundaries. The proof of the next proposition is in appendix.

\begin{prop}\label{prop:sm-fit} 
For any $t\in[0,T)$ and for $x:=b_+(t)<+\infty$ one has 
\begin{align}\label{Vx-bound00}
G'(x)\le V_x(t,x-)\le G'(x-).
\end{align}
For any $t\in[0,T)$ and for $x:=-b_-(t)>-\infty$ one has 
\begin{align}\label{Vx-bound01}
G'(x)\le V_x(t,x+)\le G'(x-).
\end{align}
\end{prop}

There are two straightforward corollaries to the above result which will be useful later in the paper. The first corollary uses that $G'$ is continuous at $x\in\RR\setminus[-a_-,a_+]$ if $\mu(\{x\})=0$.
\begin{coroll}\label{cor:smf1}
If $\mu(\{\pm b_\pm(t)\})=0$ then $V_x(t,\,\cdot\,)$ is continuous at $\pm b_\pm(t)$ so that
\begin{align*} 
V_x(t,\pm b_\pm(t))=G'(\pm b_\pm(t)).
\end{align*}
\end{coroll}
\noindent The next corollary follows by observing that, since $\mu(\{\pm\infty\})=\nu(\{\pm\infty\})=0$, then 
\begin{align*}
\lim_{x\to\pm\infty}G'(x)=\lim_{x\to\pm\infty}G'(x-)=0.
\end{align*}
Here we use the notation $V^T$ and $b_\pm^T$ for the value function \eqref{eq:V} and the corresponding optimal boundaries.
\begin{coroll}\label{cor:smf2}
Let $b_T:= b^T_+(0)$, then if $b_T<+\infty$ for all $T>0$, it holds
\begin{align*}
\lim_{T\to\infty}b_T=+\infty \implies \lim_{T\to\infty}|V^T_x(t, b_T-)-G'(b_T-)|=0.
\end{align*}
\noindent On the other hand letting $c_T:= -b^T_-(0)$, then if $c_T>-\infty$ for all $T>0$, it holds
\begin{align*}
\lim_{T\to\infty}c_T=-\infty \implies \lim_{T\to\infty}|V^T_x(t, c_T+)-G'(c_T)|=0.
\end{align*}
\end{coroll}
\noindent In the lemma below we characterise the behaviour of \eqref{Vx-bound00} and \eqref{Vx-bound01} as $t\to T$ for a fixed $T>0$ (with $V=V^T$ and $b_\pm=b^T_\pm$). The proof is given in appendix.
\begin{lemma}\label{lem:sm-fit2}
For fixed $T>0$ one has
\begin{itemize}
\item[ (i)] If $\mu(\{\hat b_+\})>0$ and/or $\mu(\{-\hat b_-\})>0$, then
\begin{align}
\label{Vx-T00} &\lim_{t\to T}V_x(t,b_+(t)-)=G'(\hat{b}_+-)\quad\textrm{and/or}\quad\lim_{t\to T}V_x(t,-b_-(t)+)=G'(-\hat{b}_-),
\end{align}
\item[(ii)] If $\mu(\{\hat b_+\})=0$ and/or $\mu(\{-\hat b_-\})=0$, then
\begin{align}
\label{Vx-T01} &\lim_{t\to T}V_x(t,b_+(t)-)=G'(\hat{b}_+)\quad\textrm{and/or}\quad\lim_{t\to T}V_x(t,-b_-(t)+)=G'(-\hat{b}_--).
\end{align}
\end{itemize}
\end{lemma}
\noindent Notice that $G'(-\hat{b}_--)<G(-\hat b_-)$ if $\hat b_-=a_-$ and $\nu(\{a_-\})>0$. Hence the notation in \eqref{Vx-T00} and \eqref{Vx-T01} is necessary.

To conclude our series of technical results concerning fine properties of $V_x$, we present a final lemma whose proof is also provided in appendix. Such result will be needed in the proof of Lemma \ref{lemma:Vt-meas} below when dealing with target measures $\mu$ supported on the positive (resp.~negative) half line. 
\begin{lemma}\label{lem:Vx-asympt}
If $\supp\{\mu\}\cap\RR_+=\emptyset$ (resp.~$\supp\{\mu\}\cap\RR_-=\emptyset$) then
\begin{align*}
\lim_{y\to+\infty}\,\,\sup_{0\le t\le T}\,\,\big|V_x(t,y)\big|=0\quad(\text{resp.}\, \lim_{y\to-\infty}\,\,\sup_{0\le t\le T}\,\,\big|V_x(t,y)\big|=0).
\end{align*}
\end{lemma}
\vspace{+6pt}

We are now going to prove that $V_t$ is continuous on $[0,T)\times\RR$.
Let us first introduce the generalised inverse of the optimal boundaries, namely let
\begin{align}\label{eq:invbarrier}
T_*(x):=
\left\{
\begin{array}{ll}
\sup\{t\in[0,T]\,:\,-b_-(t)< x \}, & x\in (-b_-(0),0)\\[+4pt]
\sup\{t\in[0,T]\,:\,b_+(t)> x\}, & x\in[0, b_+(0))\\[+4pt]
0, &\text{elsewhere}
\end{array}
\right.
\end{align}
Note that $x\in(-b_-(t),b_+(t))$ if and only if $t<T_*(x)$. Note also that $T_*$ is positive, increasing and left-continuous on $[-b_-(0),-b_-(T)]$, decreasing and right-continuous on $[b_+(T),b_+(0)]$ (hence lower semi-continuous) with $T_*(\pm b_\pm (0))=0$ if $b_\pm(0)<+\infty$.

\begin{figure}[!ht]
\label{fig:1}
\begin{center}
\includegraphics[scale=0.65]{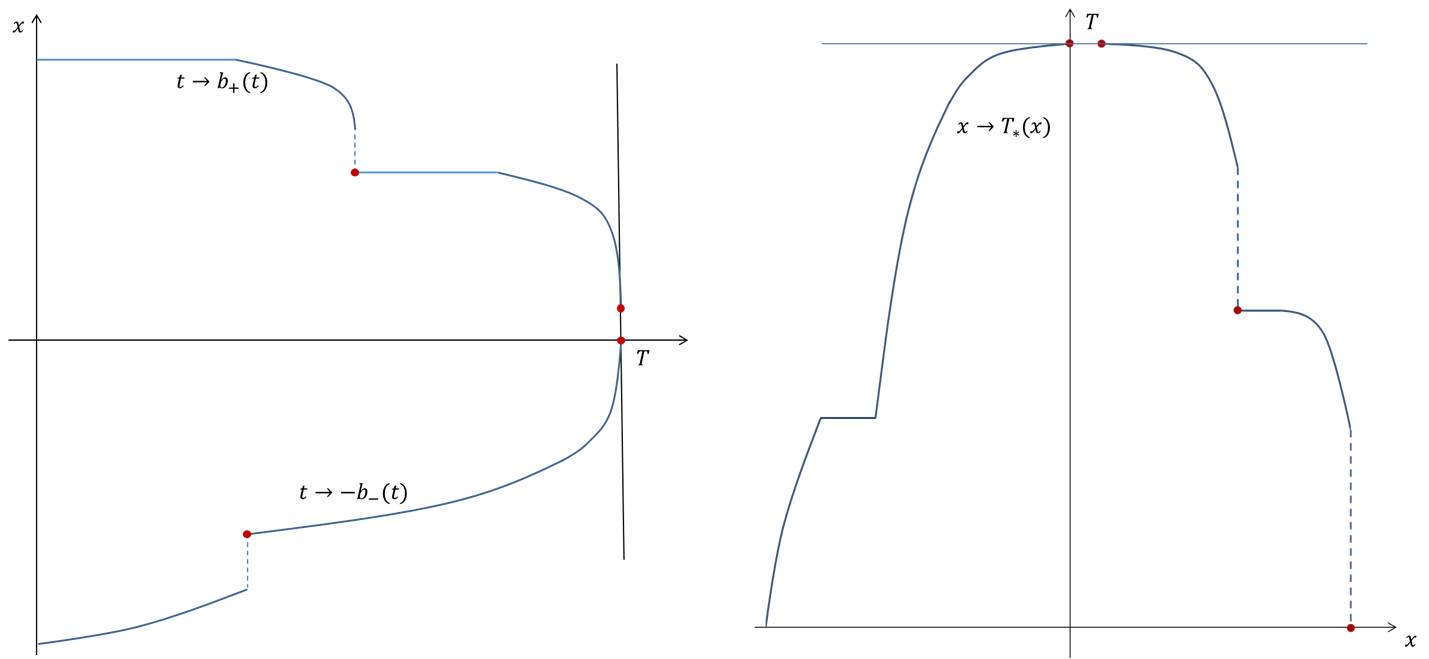}
\end{center}
\caption{A drawing of possible optimal stopping boundaries $\pm b_\pm$ (on the left) and of the corresponding generalised inverse function $T_*$ (on the right).}



\end{figure}

\begin{lemma}\label{lemma:Vt-meas}
For $h\in(0,T)$ define the measure on $\RR$
\begin{align}\label{eq:sigmah}
\sigma_h(dy):=\frac{V(T,y)-V(T-h,y)}{h}\,dy.
\end{align}
Then the family $(\sigma_h)_{h\in(0,T)}$ is a family of negative measures such that
\begin{align}\label{eq:weak}
\sigma_h(dy)\to
-\nu(dy)\quad\text{weakly as $h\to0$}
\end{align}
and $|\sigma_h(\RR)|\le L_G$ for all $h\in(0,T)$.
\end{lemma}
\begin{proof}
We start by considering $\mu_\pm>0$ so that we are in the setting of $(iii)$ in Theorem \ref{thm:OS-b}. In particular fix $\overline{h}<T$ so that $b_+$ and $b_-$ are bounded on $[T-\overline{h},T]$. Hence 
\begin{align*}
\supp\{\sigma_h\}=(-b_-(T-h),b_+(T-h))\quad\text{for all $h\in(0,\overline{h})$}
\end{align*}
because $V(T-h,y\big)=G(y)=V(T,y)$ for all $y\notin(-b_-(T-h),b_+(T-h))$.

Take an arbitrary $f\in C^2_b(\RR)$, recall \eqref{eq:invbarrier} and notice that 
\begin{align*}
&V\big(T_*(y)\vee (T-h),y\big)=V(T-h,y)=G(y)\quad\text{for $y\notin (-b_-(T-h),b_+(T-h))$,}\\[+3pt]
&V\big(T_*(y)\vee (T-h),y\big)=V(T_*(y),y)=G(y)\quad\text{for $y\in(-b_-(T-h),b_+(T-h))$.}
\end{align*}
Then we have
\begin{align*}
&\int_\RR f(y)\frac{V(T,y)-V(T-h,y)}{h}dy\nonumber\\[+3pt]
=&\int_\RR f(y)\frac{V(T,y)-V\big(T_*(y)\vee (T-h),y\big)}{h}dy\nonumber\\[+3pt]
&+\int_\RR f(y)\frac{V\big(T_*(y)\vee (T-h),y\big)-V(T-h,y)}{h}dy\\[+3pt]
=&\int^{b_+(T-h)}_{-b_-(T-h)} f(y)\frac{V\big(T_*(y),y\big)-V(T-h,y)}{h}dy.
\end{align*}
Thanks to continuity of $V$ all the integrals above are understood as integrals on open intervals, i.e.
\begin{align}\label{int-open}
\int_{-b_-(s)}^{b_+(s)}\ldots dy=\int_{\big(-b_-(s),b_+(s)\big)}\ldots dy.
\end{align}

We now recall that $V_t$ is continuous in $\CC_T$ and $V_t=-\tfrac{1}{2}V_{xx}$ in $\CC_T$. Then we use Fubini's theorem, integration by parts and \eqref{eq:pde02} to obtain
\begin{align}\label{eq:weak01}
\int^{b_+(T-h)}_{-b_-(T-h)} & f(y)\frac{V\big(T_*(y),y\big)-V(T-h,y)}{h}dy\\[+3pt]
=&\frac{1}{h}\int^{b_+(T-h)}_{-b_-(T-h)} f(y)\int^{T_*(y)}_{T-h}V_t(s,y)ds\,dy\nonumber\\[+3pt]
=&-\frac{1}{2h}\int_{T-h}^T\int_{-b_-(s)}^{b_+(s)}f(y)V_{xx}(s,y)dy\,ds\nonumber\\[+3pt]
 =&-\frac{1}{2h}\int_{T-h}^T\left[\big(f(\,\cdot\,) V_x(s,\,\cdot\,)-f'(\,\cdot\,)G(\,\cdot\,)\big|^{b_+(s)}_{-b_-(s)}+\int_{-b_-(s)}^{b_+(s)}f''(y)V(s,y)dy\right]\,ds.\nonumber
\end{align}
Notice that due to \eqref{int-open} we have 
\begin{align}\label{lim-fVx}
f(\,\cdot\,) V_x(s,\,\cdot\,)\big|^{b_+(s)}_{-b_-(s)}:=f(b_+(s)) V_x(s,b_+(s)-)-f(-b_-(s)) V_x(s,-b_-(s)+).
\end{align}

Since we are interested in the limit of the above expressions as $h\to0$ it is useful to recall Lemma \ref{lem:sm-fit2}. For simplicity we only illustrate in full details the case $\mu(\{\hat b_+\})>0$, $a_-=\hat b_-$ and $\nu(\{-a_-\})>0$ but all the remaining cases can be addressed with the same method.

Because of $\mu(\{\hat b_+\})>0$ then $a_+<\hat b_+$ (Assumption D.2) and we use (i) of Lemma \ref{lem:sm-fit2}; on the other hand for $a_-=\hat b_-$ and $\nu(\{-a_-\})>0$ we use (ii) of the same lemma. From \eqref{lim-fVx} we have
\begin{align}\label{eq:lim0}
&\lim_{s\to T}f(\,\cdot) V_x(s,\,\cdot)\big|^{b_+(s)}_{-b_-(s)}
=f(\hat{b}_+)G'(\hat{b}_+-)-f(-\hat{b}_-)G'(-a_--).
\end{align}
We take limits in \eqref{eq:weak01} as $h\to 0$, use \eqref{eq:lim0} and undo the integration by parts to obtain
\begin{align}\label{eq:weak02}
\lim_{h\to0}&\int_\RR f(y)\sigma_h(dy)\nonumber\\[+3pt]
=&-\frac{1}{2}\Big[(f G')(\hat{b}_+-)-(fG')(-a_--)-(f'G)(\hat{b}_+)+(f'G)(-a_-)+\int_{-a_-}^{\hat b_+}f''(y)G(y)dy\Big]\nonumber\\[+3pt]
=&-\frac{1}{2}\int_\RR\mathds{1}_{[-a_-,a_+]}(y)f(y)G''(dy)=-\int_\RR f(y)\nu(dy).
\end{align}
Notice that in the penultimate equality we have used that $G'(-a_-)-G'(-a_--)=2\nu(\{-a_-\})$ and 
\begin{align*}
\mathds{1}_{[-a_-,\hat b_+)}G''(dy)=2\mathds{1}_{[-a_-,a_+]}\nu(dy)
\end{align*}
(recall that $\mu((-\hat{b}_-,\hat{b}_+))=0$ and $\mu(\{-a_-\})=0$). It is important to remark that it is thanks to the fine study performed in Lemma \ref{lem:sm-fit2} that we obtain exactly the indicator of $[-a_-,a_+]$ in \eqref{eq:weak02}.

To show that $\sigma_h$ is finite on $\RR$ it is enough to take $f\equiv1$ in \eqref{eq:weak01} and notice that
\begin{align*}
\sigma_h(\RR)=&-\frac{1}{2h}\int^T_{T-h}\big(V_x(s,b_+(s))-V_x(s,-b_-(s))\big)ds\quad\text{for all $h\in(0,\overline{h})$.}
\end{align*}
From the last expression and \eqref{eq:sm-fit0} it immediately follows that $|\sigma_h(\RR)|\le L_G$.

In \eqref{eq:weak02} we have not proven weak convergence of $\sigma_h$ to $-\nu$ yet but this can now be done easily. In fact any $g\in C_b(\RR)$ can be approximated by a sequence $(f_k)_k\subset C^2_b(\RR)$ uniformly converging to $g$ on any compact. In particular, for a compact $A\supseteq \supp\{\sigma_{\overline h}\}$, and for any $\eps>0$ we can always find $K_\eps>0$ such that $\sup_A |f_k-g|\le \eps$ for all $k\ge K_\eps$. Since $\supp\{\nu\}\subseteq \supp\{\sigma_h\}\subseteq A$ for all $h\in(0,\overline h)$, the previous results give
\begin{align*}
\lim_{h\to0}\Big|\int_\RR g(y)\big(\sigma_h+\nu)(dy)\Big|&\le \lim_{h\to0}\eps\big(\big|\sigma_h(\RR)\big|+\nu(\RR)\big)+\lim_{h\to0} \Big|\int_\RR f_k(y)\big(\sigma_h+\nu)(dy)\Big|\\
&\le(1+L_G)\eps\nonumber
\end{align*}
for all $k\ge K_\eps$. Since $\eps>0$ is arbitrary \eqref{eq:weak} holds.

We now consider the case $\supp\{\mu\}\cap\RR_+=\emptyset$, i.e.~$\mu_+=-\hat{b}_-$, and $b_+(\cdot)\equiv+\infty$. Using Lemma \ref{lem:Vx-asympt} we can repeat step by step the calculations above to obtain \eqref{eq:weak02} with $\hat{b}_+=+\infty$ for any $f\in C^2_b(\RR)$ such that $f(x)\to0$ and $f'(x)G(x)\to0$ as $x\to\infty$. So it only remains to prove that the density argument holds. For that we observe that by Lemma \ref{lem:Vx-asympt} one has
\begin{align}\label{ext1}
\sigma_h\big([x,+\infty)\big)=-\frac{1}{2h}\int^T_{T-h}\int_x^\infty V_{xx}(s,y)dy\,ds=\frac{1}{2h}\int^T_{T-h}V_{x}(s,x)ds,\quad x>a_+\,,
\end{align} 
and moreover for any $\eps>0$ there exists $x_\eps>0$ such that $\big|\sigma_h\big([x,+\infty)\big)\big|\le \eps/2$ for all $x>x_\eps$. With no loss of generality we may assume that also $\nu([x_\eps,+\infty))\le \eps/2$ because $\nu$ puts no mass at infinity. Setting $A_\eps= [-b_-(T-\overline{h}),x_\eps]$, we can find a sequence $(f_k)_k\subset C^2_b(\RR)$ with $f_k(x)\to0$ and $f_k'(x)G(x)\to0$ as $x\to\infty$, and a number $K_\eps>0$ such that $\sup_{A_\eps} |f_k-g|\le \eps$ for all $k\ge K_\eps$. With no loss of generality we may also assume $\|f_k\|_\infty\le c$ for all $k$ and a given $c>0$. This gives
\begin{align*}
\Big|\int_\RR g(y)\big(\sigma_h+\nu)(dy)\Big|\le& \Big|\int_{\RR} (g-f_k)(y)\big(\sigma_h+\nu)(dy)\Big|+\Big|\int_{\RR} f_k(y)(\sigma_h+\nu)(dy)\Big|\\[+3pt]
\le &\eps \big(1+|\sigma_h(\RR)|\big)+\|g-f_k\|_{\infty}\big|(\sigma_h+\nu)([x_\eps,+\infty))\big|\\[+3pt]
&+\Big|\int_{\RR} f_k(y)(\sigma_h+\nu)(dy)\Big|.
\end{align*}
In the limit as $h\to0$ we find 
\begin{align*}
\lim_{h\to0}\Big|\int_\RR g(y)\big(\sigma_h+\nu)(dy)\Big|&\le \eps\big(1+|\sigma_h(\RR)|+\|g\|_{\infty}+c\big)
\end{align*}
and the claim follows by arbitrariness of $\eps$. The case $\supp\{\mu\}\cap\RR_-=\emptyset$ can be addressed by similar arguments and we omit the proof for brevity.
\end{proof}

Let us denote
\begin{align}\label{def:transd}
p(t,x,s,y):=\frac{1}{\sqrt{2\pi(s-t)}}e^{-\frac{(x-y)^2}{2(s-t)}},\quad\text{for $t<s$, $x,y\in\RR$}
\end{align}
the Brownian motion transition density. We can now give the main result of this section.
\begin{prop}\label{prop:cont-Vt}
It holds $V_t\in C([0,T)\times\RR)$.
\end{prop}
\begin{proof}
Continuity of $V_t$ holds separately inside $\CC_T$ and in $\DD_T$, thus it remains to verify it across the boundary of $\CC_T$. 

First we fix $t\in(0,T)$ and $x\in\RR$ such that $(t,x)\in\partial\CC_T$, and take a sequence $(t_n,x_n)_{n\in\mathbb{N}}\subset \CC_T$ such that $(t_n,x_n)\to(t,x)$ as $n\to\infty$. For technical reasons that will be clear in what follows we assume $t\le T- 2\delta$ for some arbitrarily small $\delta>0$ and with no loss of generality we also consider $t_n< T-\delta$ for all $n$. Now we aim at providing upper and lower bounds for $V_t(t_n,x_n)$ for each $n\in\mathbb{N}$. A simple upper bound follows by observing that $t\mapsto V(t,x)$ is decreasing and clearly
\begin{align}\label{eq:Vt}
V_t(t_n,x_n)\le 0\quad \text{for all $n\in\mathbb{N}$}.
\end{align}

For the lower bound we fix $n$ and take $h>0$ such that $t_n- h\ge 0$ and hence $(t_n-h,x_n)\in\CC_T$. For simplicity we denote $\tau_n=\tau_*(t_n,x_n)$ and $\tau_{n,h}:=\tau_*(t_n-h,x_n)$ as in \eqref{eq:optst} so that
$\tau_{n,h}$ is optimal for the problem with value $V(t_n-h,x_n)$.
We use the superharmonic characterisation of $V$ to obtain
\begin{align}\label{eq:Vt-01}
V(&t_n,x_n)-V(t_n-h,x_n)\\
\ge& \EE_{x_n}\left[V(t_n+\tau_{n,h}\wedge(T-t_n),B_{\tau_{n,h}\wedge(T-t_n)})-V(t_n-h+\tau_{n,h},B_{\tau_{n,h}})\right]\nonumber\\[+3pt]
=&\EE_{x_n}\left[\left(V(t_n+\tau_{n,h},B_{\tau_{n,h}})-V(t_n-h+\tau_{n,h},B_{\tau_{n,h}})\right)\mathds{1}_{\{\tau_{n,h}< T-t_n\}}\right]\nonumber\\[+3pt]
&+\EE_{x_n}\left[\left(V(T,B_{T-t_n})-V(t_n-h+\tau_{n,h},B_{\tau_{n,h}})\right)\mathds{1}_{\{\tau_{n,h}\ge T-t_n\}}\right].\nonumber
\end{align}
Observe that on the set $\{\tau_{n,h}< T-t_n\}$ it holds $V(t_n-h+\tau_{n,h},B_{\tau_{n,h}})=G(B_{\tau_{n,h}})$ and $V(t_n+\tau_{n,h},B_{\tau_{n,h}})\ge G(B_{\tau_{n,h}})$. On the other hand
\begin{align*}
\EE_{x_n}\left[ V(t_n-h+\tau_{n,h},B_{\tau_{n,h}})\big|\cF_{T-t_n}\right]=V(T-h,B_{T-t_n})\quad\text{on $\{\tau_{n,h}\ge T-t_n\}$}
\end{align*}
by the martingale property of the value function inside the continuation region. Dividing \eqref{eq:Vt-01} by $h$ and taking iterated expectations it then follows
\begin{align}\label{eq:Vt-02}
\frac{1}{h}\Big(&V(t_n,x_n)-V(t_n-h,x_n)\Big)\\
\ge& \frac{1}{h}\EE_{x_n}\left[\left(V(T,B_{T-t_n})-V(T-h,B_{T-t_n})\right)\mathds{1}_{\{\tau_{n,h}\ge T-t_n\}}\right]\nonumber\\[+3pt]
=&\EE_{x_n}\left[\frac{V(T,B_{T-t_n})-V(T-h,B_{T-t_n})}{h}\right]\nonumber\\[+3pt]
&-\EE_{x_n}\left[\mathds{1}_{\{\tau_{n,h}< T-t_n\}}\frac{V(T,B_{T-t_n})-V(T-h,B_{T-t_n})}{h}\right].\nonumber
\end{align}

Since for all $n$ we have $\delta\le T-t_n$ then $\{\tau_{n,h}\le T-t_n-\delta\}\subseteq \{\tau_{n,h}< T-t_n\}$ and since
$V(T,B_{T-t_n})-V(T-h,B_{T-t_n})\le0$ we obtain
\begin{align}\label{eq:Vt-03}
-&\EE_{x_n}\left[\mathds{1}_{\{\tau_{n,h}< T-t_n\}}\frac{V(T,B_{T-t_n})-V(T-h,B_{T-t_n})}{h}\right]\\[+3pt]
\ge& -\EE_{x_n}\left[\mathds{1}_{\{\tau_{n,h}\le T-t_n-\delta\}}\frac{V(T,B_{T-t_n})-V(T-h,B_{T-t_n})}{h}\right]\nonumber\\[+3pt]
=&-\EE_{x_n}\left[\mathds{1}_{\{\tau_{n,h}\le T-t_n-\delta\}}\EE_{B_{\tau_{n,h}}}\left(\frac{V(T,B_{T-t_n-\tau_{n,h}})-V(T-h,B_{T-t_n-\tau_{n,h}})}{h}\right)\right]\nonumber
\end{align}
where the last expression follows by the strong Markov property. Recalling now \eqref{eq:sigmah} and \eqref{def:transd}, and using \eqref{eq:Vt-02} and \eqref{eq:Vt-03} we obtain
\begin{align}\label{eq:Vt-04}
&\frac{V(t_n,x_n)-V(t_n-h,x_n)}{h}\ge\int_\RR f_{n,h}(y)\sigma_h(dy),
\end{align}
where
\begin{align}\label{eq:Vt-05}
f_{n,h}(y):=p(0,x_n,T-t_n,y)-\EE_{x_n}\left[\mathds{1}_{\{\tau_{n,h}\le T-t_n-\delta\}}p(0,B_{\tau_{n,h}},T-t_n-\tau_{n,h},y)\right].
\end{align}

Notice that $|f_{n,h}(y)|\le C$ for some constant independent of $n$ and $h$ (this is easily verified since $T-t_n-\tau_{n,h}\ge\delta$ in the second term of \eqref{eq:Vt-05}). Recalling Lemma \ref{cor:contST-t} it is not hard to verify that for any $(y_h)_{h>0}\subset\RR$ such that $y_h\to y\in\RR$ as $h\to0$ it holds
\begin{align*}
\lim_{h\to 0}f_{n,h}(y_h)\ge f_n(y):= p(0,x_n,T-t_n,y)-\EE_{x_n}\left[\mathds{1}_{\{\tau_{n}< T-t_n-\delta\}}p(0,B_{\tau_{n}},T-t_n-\tau_{n},y)\right],
\end{align*}
where we have used that $\lim_{h\to 0}\mathds{1}_{\{\tau_{n,h}\le T-t_n-\delta\}}\le\mathds{1}_{\{\tau_{n}\le T-t_n-\delta\}}$ since $\tau_{n,h}\downarrow\tau_n$.
Moreover, Lemma \ref{lemma:Vt-meas} implies that $\big(\sigma_h(dy)/\sigma_h(\RR)\big)_{h\in(0,\overline{h})}$ forms a weakly converging family of probability
measures. Therefore we can use a continuous mapping theorem as in \cite[Ch.~4, Thm.~4.27]{Kal} to take limits in \eqref{eq:Vt-04} as $h\to 0$ and get
\begin{align*}
V_t(t_n,x_n)\ge&\lim_{h\to0}\int_\RR f_{n,h}(y)\sigma_h(dy)=-\int_\RR f_n(y)\nu(dy).
\end{align*}

Finally we take limits as $n\to\infty$ in the last expression and we use dominated convergence, the fact that $\tau_n\to0$ as $n\to\infty$ (see Lemma \ref{lem:contST}) and the upper bound \eqref{eq:Vt}, to obtain
\begin{align*}
\lim_{n\to\infty} V_t(t_n,x_n)=0.
\end{align*}
Since the sequence $(t_n,x_n)$ was arbitrary the above limit implies continuity of $V_t$ at $(t,x)\in\partial\CC_T\cap\{t<T\}$.
\end{proof}
\vspace{+6pt}

It is a remarkable fact that in this context continuity of the time derivative $V_t$ holds at all points of the boundary regardless of whether or not the $x$-derivative $V_x$ is continuous there. As a consequence of the above theorem and of \eqref{eq:pde01} we also obtain
\begin{coroll}\label{cor:contVxx}
For any $\eps>0$ it holds that $V_x$ and $V_{xx}$ are continuous on the closure of $\CC_T\cap\{t\le T-\eps\}$.
In particular for any $(t,x)\in\partial\CC_T$ and any sequence $(t_n,x_n)_{n\in\mathbb{N}}\subset\CC_T$ such that $(t_n,x_n)\to(t,x)$ as $n\to\infty$, it holds
\begin{align*}
\lim_{n\to\infty}V_{xx}(t_n,x_n)=0.
\end{align*}
\end{coroll}

We conclude the section with a technical lemma that will be useful in the rest of the paper.
\begin{lemma}\label{lemma:VtdxT}
For any $f\in C_b(\RR)$ one has
\begin{align}\label{eq:limVt}
\lim_{t\uparrow T}\int_\RR f(x)V_t(t,x)dx=-\int_\RR f(x)\nu(dx)
\end{align}
i.e.~it holds $V_t(t,x)dx\to -\nu(dx)$ weakly as a measure, in the limit as $t\uparrow T$.
\end{lemma}
\begin{proof}
The proof is very similar to that of Lemma \ref{lemma:Vt-meas}. It suffices to prove the claim for $\mu_\pm>0$ and $f\in C^2_b(\RR)$ since arguments as in the final part of the proof of Lemma \ref{lemma:Vt-meas} allow us to extend the result to $f\in C_b(\RR)$ and any $\mu_\pm$.

We take $\overline{h}>0$ as in the proof of Proposition \ref{prop:sm-fit} and we let $A\subset\RR$ be an open bounded interval such that $[-b_-(T-\overline{h}),b_+(T-\overline{h})]\subset A$. 
Then for any $f\in C^2_b(\RR)$, $t\in(T-\overline{h},T)$ we use Proposition \ref{prop:cont-Vt} along with \eqref{eq:pde01} and \eqref{eq:pde02} to obtain
\begin{align*}
\int_A f(y)V_t(t,y)dy=&-\tfrac{1}{2}\int^{b_+(t)}_{-b_-(t)} f(y)V_{xx}(t,y)dy\\
=&-\tfrac{1}{2}\left[\big(f(\,\cdot\,)V_x(t,\,\cdot\,)-f'(\,\cdot\,)G(\,\cdot\,)\big|^{b_+(t)}_{-b_-(t)}+\int_Af''(y)V(t,y)dy\right].
\end{align*}
Taking limits as $t\to T$ and arguing as in \eqref{eq:weak02} we obtain \eqref{eq:limVt}. 
\end{proof}


\section{The Skorokhod embedding}\label{sec:sk}

In this section we will show that the optimal boundaries $b_\pm$ found in Theorem \ref{thm:OS} are the boundaries of the time reversed Rost's barrier associated to $\mu$. 

Here we recall the notation introduced in Section \ref{sec:sett} and let $s_-$ and $s_+$ be the reversed boundaries from Definition \ref{def:spm}. We denote 
\begin{align*}
&\CC^-_\infty:=\big\{(t,x)\in[0,+\infty)\times\RR\,:\,x\in\big(-s_-(t),s_+(t)\big)\big\},\\
&\DD^-_\infty:=\big\{(t,x)\in[0,+\infty)\times\RR\,:\,x\in\big(-\infty,-s_-(t)]\cup[s_+(t),+\infty\big)\big\},
\end{align*}
again with the convention $(-\infty,-\infty]=[+\infty,+\infty)=\emptyset$.

Arguing as in \eqref{eq:invbarrier} we introduce the (generalised) inverse of $s_\pm$ defined by
\begin{align}\label{eq:invbarrier00}
\varphi(x):=
\left\{
\begin{array}{ll}
\inf\{t\ge 0\,:\,-s_-(t)< x \}, & x\le -s_-(0)\\[+4pt]
0, & x\in(-s_-(0),s_+(0))\\[+4pt]
\inf\{t\ge0\,:\,s_+(t)> x\}, & x\ge s_+(0)
\end{array}
\right.
\end{align}
Notice that $x\in(-s_-(t),s_+(t))$ if and only if $\varphi(x)<t$ and note also that for each $T>0$ it holds (see \eqref{eq:invbarrier})
\begin{align*}
T_*(x)=T-\varphi(x),\qquad\text{for}\: x\in[-s_-(T),s_+(T)].
\end{align*}
It is not hard to see that $\varphi$ is positive, decreasing left-continuous on $\RR_-$ and increasing right-continuous on $\RR_+$ (hence upper semi-continuous).

\begin{figure}[!ht]
\label{fig:2}
\begin{center}
\includegraphics[scale=0.65]{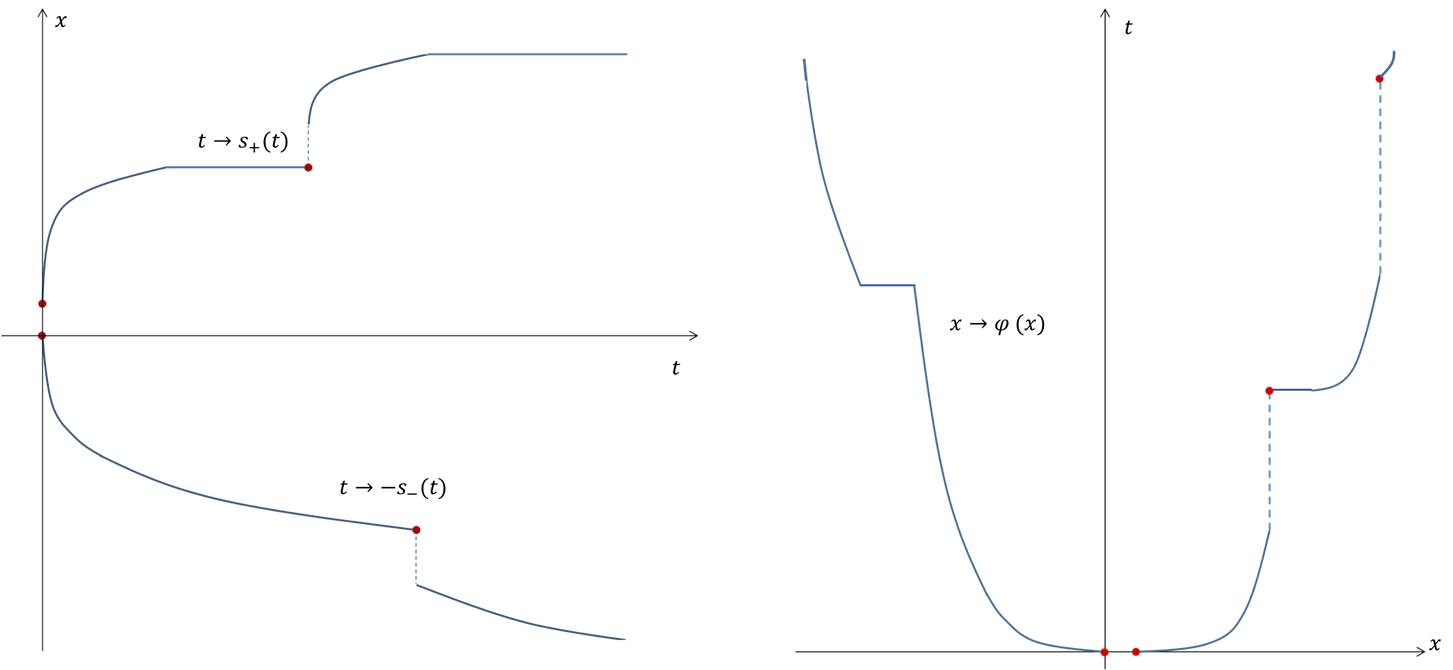}
\end{center}
\caption{A drawing of possible reversed boundaries $s_+$ and $-s_-$ (on the left) and of the corresponding generalised inverse function $\varphi$ (on the right).}



\end{figure}

Our first step is to use stochastic calculus to find a probabilistic representation of $V_t$. 
Let us start by introducing some notation. Along with the Brownian motion $B$ we consider another Brownian motion $W:=(W_t)_{t\ge0}$ independent of $B$ and we denote $(\cF^W_t)_{t\ge0}$ the filtration generated by $W$ and augmented with $\PP$-null sets. Recalling $\tau_*$, $\tilde \tau_*$ and \eqref{st-interior}, we now introduce similar concepts relative to the sets $\CC^-_\infty$ and $\DD^-_\infty$. For $(t,x)\in\RR_+\times\RR$ we now set
\begin{align}
\label{eq:tau-} &\tau_-(t,x):=\inf\big\{u>0\,:\,x+W_u\notin\big(-s_-(t+u),s_+(t+u)\big)\big\}\\[+4pt]
\label{eq:tau-2} &\tilde{\tau}_-(t,x):=\inf\big\{u>0\,:\,x+W_u\notin\big[-s_-(t+u),s_+(t+u)\big]\big\}.
\end{align}
It is clear that $\tau_-$ and $\tilde{\tau}_-$ are $(\cF^W_t)$-stopping times. 
Moreover in \cite{Cox-Pe13} (see eq.~(2.9) therein) one can find an elegant proof of the fact that\footnote{To avoid confusion note that in \cite{Cox-Pe13} our functions $s_+$ and $-s_-$ are denoted respectively $b$ and $c$.}
\begin{align}\label{eq:equivtaus}
\PP_{t,x}(\tau_-=\tilde{\tau}_-)=1\quad\text{for all $(t,x)\in[0,+\infty)\times\RR$}.
\end{align}
The latter plays a similar role to \eqref{st-interior} in the case of the sets $\CC^-_\infty$ and $\DD^-_\infty$.
In what follows, and in particular for Lemma \ref{lemma:Hunt}, we will find sometimes convenient to use $\tilde{\tau}_-$ instead of $\tau_-$ to carry out our arguments of proof.

The stopping times $\tau_-$ and $\tilde\tau_-$ are introduced in order to link $V_t$ to the transition density of the process $(t,W_t)$ killed upon leaving the set $\CC^-_\infty$. This is done in Proposition \ref{prop:Ut}. The latter is then used to prove that $\DD_\infty^-$ is indeed the Rost's barrier (see the proof of Theorem \ref{thm:sk} provided below). 

From now on we denote $p^\CC(t,x,s,y)$, $s>t$, the transition density associated with the law $\PP_{t,x}(B_s\in dy\,,\,s\le\tau_*)$ of the Brownian motion killed at $\tau_*$. Similarly we denote $p^\CC_-(t,x,s,y)$, $s>t$, the transition density associated with the law $\PP_{t,x}(W_s\in dy\,,\,s\le\tau_-)$ of $W$ killed at $\tau_-$. It is well known that
\begin{align}
\label{eq:pC}&p^\CC(t,x,s,y)=p(t,x,s,y)-\EE_{t,x}\mathds{1}_{\{s>\tau_*\}}p(\tau_*,B_{\tau_*},s,y)
\end{align}
for $(t,x),(s,y)\in\CC_T$ and
\begin{align}
\label{eq:pC-}&p^\CC_-(t,x,s,y)=p(t,x,s,y)-\EE_{t,x}\mathds{1}_{\{s>\tau_-\}}p(\tau_-,W_{\tau_-},s,y)
\end{align}
for $(t,x),(s,y)\in\CC^-_\infty$ (see e.g.~\cite[Ch.~24]{Kal}).

The next lemma provides a result which can be seen as an extension of Hunt's theorem as given in \cite[Ch.~24, Thm.~24.7]{Kal} to time-space Brownian motion. Although such result seems fairly standard we could not find a precise reference for its proof in the time-space setting and for the sake of completeness we provide it in the appendix.
\begin{lemma}\label{lemma:Hunt}
For all $0\le t<s\le T$ and $x\in(-b_-(t),b_+(t))$, $y\in(-b_-(s),b_+(s))$, it holds $p^{\CC}(t,x,s,y)=p^{\CC}_-(T-s,y,T-t,x)$.
\end{lemma}
\vspace{+6pt}

For future frequent use we also define
\begin{align}
\label{eq:U}U^T(t,x):=&V^T(t,x)-G(x),\qquad (t,x)\in[0,T]\times\RR
\end{align}
then $U^T\in C([0,T]\times\RR)$ and \eqref{eq:pde01}--\eqref{eq:pde02} imply
\begin{align}
\label{eq:pdeU01}&\big(U^T_t+\tfrac{1}{2}U^T_{xx}\big)(t,x)=-(\nu-\mu)(dx), & x\in\big(-b_-(t),b_+(t)\big),\:t\in[0,T)\\[+3pt]
\label{eq:pdeU02}&U^T(t,x)=0, &x\in(-\infty,-b_-(t)]\cup[b_+(t),\infty),\: t\in[0,T)\\[+3pt]
\label{eq:pdeU03}&U^T(T,x)=0, & x\in\RR
\end{align}
where the first equation holds in the sense of distributions, and in the second one we shall always understand $(-\infty,-\infty]=[+\infty,+\infty)=\emptyset$.

We can now use Lemma \ref{lemma:Hunt} to find a convenient expression for $U^T_t$ in terms of $p^\CC_-$. 
\begin{prop}\label{prop:Ut}
Fix $T>0$ and denote $U=U^T$ for simplicity (see \eqref{eq:U}). Then $U_t\in C([0,T)\times\RR)$ and it solves
\begin{align}
\label{eq:pdeUt01}&\big((U_t)_t+\tfrac{1}{2}(U_t)_{xx}\big)(t,x)=0, &(t,x)\in\CC_T\\[+5pt]
\label{eq:pdeUt02}&U_t(t,x)=0, &(t,x)\in\partial \CC_T\cap\{t<T\}\\
\label{eq:pdeUt03}&{ \lim_{t\uparrow T}\int_{\RR}f(x)U_t(t,x)dx=-\int_{\RR}f(x)\nu(dx)}, &\text{for all $f\in C_b(\RR)$}.
\end{align}

Moreover the function $U_t$ has the following representation
\begin{align}\label{eq:Ut-trans}
-U_t(t,x)=\int_\RR p^{\CC}(t,x,T,y)\nu(dy)=
\int_\RR p^{\CC}_-(0,y,T-t,x)\nu(dy),\quad(t,x)\in[0,T)\times\RR.
\end{align}
\end{prop}

\begin{proof} The proof is divided in a number of steps.
\vspace{+6pt}

\noindent\emph{Step 1}. We have already shown in Proposition \ref{prop:cont-Vt} that $V_t$ is continuous on $[0,T)\times\RR$ and equals zero along the boundary of $\CC_T$ for $t<T$. Moreover Lemma \ref{lemma:VtdxT} implies the terminal condition \eqref{eq:pdeUt03}.
In the interior of $\CC_T$ one has $V_t\in C^{1,2}$ by standard results on Cauchy-Dirichlet problems (see for instance \cite[Ch.~3, Thm.~10]{Fri08}). It then follows that $U_t$ solves \eqref{eq:pdeUt01} by differentiating \eqref{eq:pdeU01} with respect to time.
\vspace{+6pt}

\noindent \emph{Step 2}. We now aim at showing \eqref{eq:Ut-trans}. For $(t,x)$ in the interior of $\DD_T$ the result is trivial since $U_t=0$ therein. Hence we prove it for $(t,x)\in\CC_T$ and the extension to $\partial\CC_T$ will follow since $U_t$ is continuous on $[0,T)\times\RR$.

In what follows we fix $(t,x)\in\CC_T$ and set $\tau_*=\tau_*(t,x)$. For $\eps>0$ we use It\^o's formula, \eqref{eq:pdeUt01}--\eqref{eq:pdeUt03}, strong Markov property and the definition of $p^\CC$ to obtain
\begin{align*}
-U_t(t,x)=&-\EE_x U_t(t+\tau_*\wedge(T-t-\eps),B_{\tau_*\wedge(T-t-\eps)})\\
=&-\EE_x U_t(T-\eps,B_{T-t-\eps})\mathds{1}_{\{\tau_*\ge T-t-\eps\}}\nonumber\\
=&-\int_\RR U_t(T-\eps,y)p^\CC(t,x,T-\eps,y)dy\nonumber
\end{align*}
Now we want to pass to the limit as $\eps\to0$ and use Lemma \ref{lemma:VtdxT} and a continuous mapping theorem to obtain \eqref{eq:Ut-trans}. This is accomplished in the next two steps.
\vspace{+6pt}

\noindent\emph{Step 3}. First we assume that $\hat{b}_\pm>a_\pm$. Note that from \eqref{eq:pC} one can easily verify that $(s,y)\mapsto p^\CC(t,x,s,y)$ is continuous at all points in the interior of $\CC_T$ by simple estimates on the Gaussian transition density. Therefore for any $y\in[-a_-,a_+]$, any sequence $(\eps_j)_{j\in\NN}$ with $\eps_j\to0$ as $j\to\infty$, and any sequence $(y_{\eps_j})_{j\in\mathbb{N}}$ converging to $y$ as $j\to\infty$ there is no restriction in assuming $(T-\eps_j,y_{\eps_j})\in\CC_T$ so that $p^\CC(t,x,T-\eps_j,y_{\eps_j})\to p^\CC(t,x,T,y)$ as $j\to\infty$.
Hence taking limits as $\eps\to0$ and using \eqref{eq:limVt} and a continuous mapping theorem as in \cite[Ch.~4, Thm.~4.27]{Kal} we obtain
\begin{align*}
-U_t(t,x)
=\int_\RR p^\CC(t,x,T,y)\nu(dy)=\int_\RR p^\CC_-(0,y,T-t,x)\nu(dy)
\end{align*}
where the last equality follows from Lemma \ref{lemma:Hunt}.
\vspace{+6pt}

\noindent\emph{Step 4}. Here we consider the opposite situation to step 3 above, i.e.~the case $\hat b_\pm=a_\pm$. 
For arbitrary $\delta>0$ we introduce the approximation
\begin{align*}
F^\delta_\mu(x):=\left\{
\begin{array}{ll}
F_\mu(x), & x\in(-\infty,-\hat{b}_--\delta]\\[+3pt]
F_\mu(-\hat{b}_--\delta), & x\in(-\hat{b}_--\delta,\hat{b}_++\delta)\\[+3pt]
F_\mu(x)-\big[F_\mu((\hat{b}_++\delta)-)-F_\mu(-\hat{b}_--\delta)\big], & x\in[\hat{b}_++\delta,\infty)
\end{array}
\right.
\end{align*}
which is easily verified to fulfil
\begin{align}\label{eq:Fdelta01}
\lim_{\delta\to0} \sup_{x\in\RR}\big|F^\delta_\mu(x)-F_\mu(x)\big|=0
\end{align}
since $F_\mu$ is continuous at $\pm\hat{b}_\pm$ by Assumption D.2.
Moreover for $\mu^\delta(dx):=F^\delta(dx)$ we have
\begin{align*}
\mu^\delta(dx)=\left\{
\begin{array}{ll}
\mu(dx), & x\in(-\infty,-\hat{b}_--\delta]\cup[\hat{b}_++\delta,+\infty)\\[+5pt]
0,& x\in(-\hat{b}_--\delta,\hat{b}_++\delta)
\end{array}
\right.
\end{align*}

Associated to each $F^\delta_\mu$ we consider an approximating optimal stopping problem with value function $V^\delta$. The latter is defined as in \eqref{eq:V} with $G$ replaced by $G^\delta$, and $G^\delta$ defined as in \eqref{eq:G} but with $F^\delta_\mu$ in place of $F_\mu$. It is clear that the analysis carried out in Theorem \ref{thm:OS-b} and Proposition \ref{prop:sm-fit} for $V$ and $G$ can be repeated with minor changes when considering $V^\delta$ and $G^\delta$. Indeed the only conceptual difference between the two problems is that $F^\delta_\mu$ does not describe a probability measure on $\RR$ being in fact $\mu^\delta(\RR)<1$. 

In particular the continuation set for the approximating problem, i.e.~the set where $V^\delta>G^\delta$, is denoted by $\CC^\delta_T$ and there exists two right-continuous, decreasing, positive functions of time $b_\pm^\delta$ with $b_\pm^\delta(T-)=\hat{b}_\pm+\delta$ such that
\begin{align*}
\CC_T^\delta:=\big\{(t,x)\in[0,T)\times\RR\,:\,x\in\big(-b^\delta_-(t),b^\delta_+(t)\big)\big\}.
\end{align*}

It is clear from the definition of $F^\delta_\mu$ that for any Borel set $A\in\RR$ it holds $\mu^\delta(A)\le \mu^{\delta'}(A)$ if $\delta'<\delta$. Hence for $\delta'<\delta$, $(t,x)\in[0,T)\times\RR$ we obtain the following key inequality
\begin{align*}
V^\delta(t,x)-G^\delta(x)=&\sup_{0\le\tau\le T-t}\EE_x\int_\RR L^z_\tau(\nu-\mu^\delta)(dz)\\[+3pt]
\ge&\sup_{0\le\tau\le T-t}\EE_x\int_\RR L^z_\tau(\nu-\mu^{\delta'})(dz)\nonumber\\[+3pt]
=&V^{\delta'}(t,x)-G^{\delta'}(x)\nonumber
\end{align*}
by It\^o-Tanaka-Meyer formula. The above also holds if we replace $V^{\delta'}-G^{\delta'}$ by $V-G$ and it implies that the family of sets $\big(\CC^\delta_T\big)_{\delta>0}$ decreases as $\delta\downarrow 0$ with $\CC^\delta_T\supseteq\CC_T$ for all $\delta>0$. We claim that
\begin{align}\label{eq:convCd}
\lim_{\delta\to0}\CC^\delta_T=\CC_T\quad\text{and}\quad\lim_{\delta\to0}b^\delta_\pm(t)=b_\pm(t)\quad\text{for all $t\in[0,T)$.}
\end{align}
The proof of the above limits follows from standard arguments and is given in appendix where it is also shown that
\begin{align}\label{eq:Vdelta01}
\lim_{\delta\to0} \sup_{(t,x)\in[0,T)\times K}\big|V^\delta(t,x)-V(t,x)\big|=0, \quad\text{$K\subset\RR$ compact}.
\end{align}

Now for each $\delta>0$ we can repeat the arguments that we have used above in this section and in Section \ref{sec:sett} to construct a set $\CC^{\delta,-}_\infty$ which is the analogue of the set $\CC^-_\infty$. All we need to do for such construction is to replace the functions $s_+$ and $s_-$ by their counterparts $s^\delta_+$ and $s^\delta_-$ which are obtained by pasting together the reversed boundaries $s^{\delta,n}_\pm(t):=b^{\delta,T_n}_\pm(T_n-t)$, $t\in(0,T_n]$ (see Definition \ref{def:spm} and the discussion preceding it).

As in \eqref{eq:optst} and \eqref{def:tau-t} we define by $\tau^\delta_*$ the first time the process $(t,B_t)_{t\ge0}$ leaves $\CC^\delta_T$ and by $\tilde \tau^\delta_*$ the first time $(t,B_t)_{t\ge0}$ leaves the closure of $\CC^\delta_T$. Similarly to \eqref{eq:tau-} and \eqref{eq:tau-2} we also denote by $\tau^\delta_-$ and $\tilde\tau^\delta_-$ the first strictly positive times the process $(W_t)_{t\ge0}$ leaves $(-s^\delta_-(t),s^\delta_+(t))$ and $[-s^\delta_-(t),s^\delta_+(t)]$, $t>0$, respectively. It holds again, as in \eqref{eq:equivtaus}, that 
\begin{align}\label{eq:equivtaus2}
\text{$\PP_{t,x}(\tau^\delta_-=\tilde\tau^\delta_-)=1$ for all $(t,x)\in[0,+\infty)\times\RR$.}
\end{align}

It is clear that $\tau^\delta_-$ decreases as $\delta\to0$ (since $\delta\mapsto\CC^\delta_T$ is decreasing) and $\tau^\delta_-\ge \tau_-$, $\PP$-a.s.~for all $\delta>0$. We show in appendix that in fact
\begin{align}\label{eq:convTd}
\lim_{\delta\to0}\tau^\delta_-=\tau_-,\quad\text{$\PP$-a.s.}
\end{align}

The same arguments used to prove Proposition \ref{prop:cont-Vt} (up to a refinement of Lemmas \ref{lem:Vx-asympt} and \ref{lemma:Vt-meas} which we discuss in the penultimate section of the appendix) can now be applied to show that $V^\delta_t$ is continuous on $[0,T)\times\RR$ and $V^\delta_t=0$ outside of $\CC^\delta_T\cap\{t<T\}$. Therefore, for fixed $\delta>0$, we can use the arguments of step 1, step 2 and step 3 above since $\hat{b}_\pm+\delta>a_\pm$ and obtain
\begin{align}\label{eq:FK2d}
-U^\delta_t(t,x)
=\int_\RR p^{\CC,\delta}(t,x,T,y)\nu(dy)=\int_\RR p^{\CC,\delta}_-(0,y,T-t,x)\nu(dy)
\end{align}
where obviously the transition densities $p^{\CC,\delta}$ and $p^{\CC,\delta}_-$ have the same meaning of $p^{\CC}$ and $p^{\CC}_-$ but with the sets $\CC_T$ and $\CC^-_\infty$ replaced by $\CC^\delta_T$ and $\CC^{\delta,-}_\infty$, respectively. Note that $U^\delta_t\le0$, then for fixed $t\in[0,T)$ the expression above implies (see \eqref{eq:pC} and \eqref{eq:pC-})
\begin{align*}
\sup_{x\in\RR}\big|U^\delta_t(t,x)\big|\le\sup_{x\in\RR}\int_\RR p(0,y,T-t,x)\nu(dy)<+\infty\quad\text{for all $\delta>0$}
\end{align*}
and therefore there exists $g\in L^\infty(\RR)$ such that $U^\delta_t(t,\,\cdot\,)$ converges along a subsequence to $g$ as $\delta\to0$ in the weak* topology relative to $L^\infty(\RR)$. Moreover since \eqref{eq:Vdelta01} holds and the limit is unique, it must also be $g(\,\cdot\,)=U_t(t,\,\cdot\,)$.

Now, for an arbitrary Borel set $B\subseteq [-s_-(T-t),s_+(T-t)]$, \eqref{eq:FK2d} gives
\begin{align*}
-\int_B U^\delta_t(t,x)dx=\int_\RR \PP_y(W_{T-t}\in B,\:T-t\le\tau^\delta_-)\nu(dy).
\end{align*}
We take limits in the above equation as $\delta\to0$ (up to selecting a subsequence), we use dominated convergence and \eqref{eq:convTd} for the right-hand side, and weak* convergence of $U^\delta_t$ for the left-hand side, and obtain
\begin{align*}
-\int_B U_t(t,x)dx=\int_\RR \PP_y(W_{T-t}\in B,\:T-t\le\tau_-)\nu(dy).
\end{align*}

Finally, since $B$ is arbitrary we can conclude that \eqref{eq:Ut-trans} holds in general.
\vspace{+6pt}

After step 3 and 4 the remaining intermediate cases are: (i) $\hat b_+=a_+$ and $\hat b_->a_-$, and (ii) $\hat b_-=a_-$ and $\hat b_+>a_+$. These may be addressed by a simple combination of the methods developed in steps 3 and 4 and we omit further details. 
\end{proof}
\vspace{+6pt}

Now we are ready to prove the main result of this section, i.e.~Theorem \ref{thm:sk}, whose statement we recall for convenience.
\vspace{+8pt}

\noindent\textbf{Theorem 2.3}\emph{
Let $W^\nu:=(W^\nu_t)_{t\ge0}$ be a standard Brownian motion with initial distribution $\nu$ and define
\begin{align}
\sigma_*:=\inf\big\{t>0\,:\,W^\nu_t\notin \big(-s_-(t),s_+(t)\big)\big\}.
\end{align}
Then it holds
\begin{align}\label{eq:skb}
\EE f(W^\nu_{\sigma_*})\mathds{1}_{\{\sigma_*<+\infty\}}=\int_\RR f(y)\mu(dy),\quad\text{for all $f\in C_b(\RR)$.}
\end{align}
}
\begin{proof}
We start by recalling that since $s_\pm(T)=b^T_\pm(0)$, then Proposition \ref{prop:s-lim} and Corollary \ref{cor:suppmu} imply that
\begin{align}\label{eq:b1}
\lim_{T\to\infty}s_\pm(T)=b_\pm^\infty\ge \mu_\pm
\end{align}
where we also recall that $\mu_\pm$ are the endpoints of $\supp\,\mu$ (see \eqref{not:mu}). Notice that by monotonicity of the boundaries if $s_+(t_0)=+\infty$, then $s_+(t)=+\infty$ for $t\ge t_0$ and the same is true for $s_-$.
 
Fix an arbitrary time horizon $T$ and denote $U^T=U$ as in \eqref{eq:U}. Throughout the proof all Stieltjes integrals with respect to measures $\nu$ and $\mu$ on $\RR$ are taken on open intervals, i.e.
$$\int_a^b\ldots\,\,=\int_{(a,b)}\ldots\qquad \text{for $a<b$.}$$

Let $f\in C^2_b(\RR)$ and consider the sequence $(f_n)_{n\ge0}\subset C^2_b(\RR)$ with $f_n(x)=f(x)$ for $|x|\le n$ and $f_n(x)=0$ for $|x|\ge n+1$. Notice that 
\begin{align}\label{eq:lebs}
\EE f(W^\nu_{\sigma_*\wedge T})=\lim_{n\to\infty}\EE f_n(W^\nu_{\sigma_*\wedge T})
\end{align}
by dominated convergence and the fact that $f_n\to f$ pointwise at all $x\in\RR$.

Now, for arbitrary $n$ a straightforward application of It\^o's formula gives
\begin{align}\label{eq:Ito}
\EE f_n(W^\nu_{\sigma_*\wedge T})=&\int_\RR f_n(y)\nu(dy)+\frac{1}{2}\EE\int_0^{\sigma_*\wedge T}{f_n''}(W^\nu_u)du\\
=&\int_\RR f_n(y)\nu(dy)+\frac{1}{2}\int_0^{T}{\EE\mathds{1}_{\{u\le\sigma_*\}}f_n''(W^\nu_u)du}\nonumber
\end{align}
Notice that $\sigma_*=\tau_-=\tilde{\tau}_-$ (see \eqref{eq:tau-}--\eqref{eq:equivtaus}) up to replacing the initial condition in the definitions of $\tau_-$ and $\tilde{\tau}_-$ by an independent random variable with distribution $\nu$. Recall the probabilistic representation \eqref{eq:Ut-trans} of $U_t$. Then we observe that for $u>0$
\begin{align*}
\EE\mathds{1}_{\{u\le\sigma_*\}}f_n''(W^\nu_u)=&\int_\RR f_n''(y)\Big(\int_\RR p^\CC_-(0,x,u,y)\nu(dx)\Big)dy\\
=&-\int^{s_+(u)}_{-s_-(u)} U_t(T-u,y)f_n''(y)dy\nonumber
\end{align*}
by \eqref{eq:pdeUt02}. An application of Fubini's theorem and the fact that $y\in(-s_-(u),s_+(u))\iff u>\varphi(y)$ (see \eqref{eq:invbarrier00}) gives
\begin{align}\label{eq:fub01}
\int_0^{T}\EE\mathds{1}_{\{u\le\sigma_*\}}f_n''(W^\nu_u)\,du=&-\int_0^T\Big(\int_\RR\mathds{1}_{\{y\in(-s_-(u),s_+(u))\}}U_t(T-u,y)f_n''(y)dy\Big)du\\
=&-\int_\RR f_n''(y)\Big(\int_0^T\mathds{1}_{\{\varphi(y)<u\}}U_t(T-u,y)du\Big)dy\nonumber\\
=&\int_\RR f_n''(y)\Big(U(0,y)-U(T-\varphi(y),y)\Big)dy\nonumber\\
=&\int_\RR f_n''(y)U(0,y)dy\nonumber
\end{align}
where in the last line we have also used that $(T-\varphi(y),y)=(T_*(y),y)\in\partial \CC_T$ and $U|_{\partial \CC_T}=0$ (see \eqref{eq:pdeU02}). Hence from \eqref{eq:Ito} and \eqref{eq:fub01}, and using that $U(0,y)=0$ for $y\notin(-s_-(T),s_+(T))$, we conclude
\begin{align}\label{eq:intbypart00}
\EE f_n(W^\nu_{\sigma_*\wedge T})=&\int_\RR f_n(y)\nu(dy)+\frac{1}{2}\int^{s_+(T)}_{-s_-(T)}f_n''(y)U(0,y)dy.
\end{align}
Notice that the last term above makes sense even if $s_\pm(T)=+\infty$, because $f_n$ is supported on a compact.

The left hand side of \eqref{eq:intbypart00} has an alternative representation and in fact one has
\begin{align*}
\EE f_n(W^\nu_{\sigma_*\wedge T})=&\EE \mathds{1}_{\{T\le\sigma_*\}}f_n(W^\nu_{T})+\EE \mathds{1}_{\{\sigma_*< T\}}f_n(W^\nu_{\sigma_*})\nonumber\\
=&\int_{-s_-(T)}^{s_+(T)}\left(\int_\RR f_n(y) p^\CC_-(0,x,T,y)\nu(dx)\right)dy+\EE \mathds{1}_{\{\sigma_*< T\}}f_n(W^\nu_{\sigma_*}).
\end{align*}
By using \eqref{eq:Ut-trans} once more we obtain
\begin{align}\label{eq:alt01}
&\int_{-s_-(T)}^{s_+(T)}\left(\int_\RR f_n(y) p^\CC_-(0,x,T,y)\nu(dx)\right)dy=-\int_{-s_-(T)}^{s_+(T)}f_n(y)U_t(0,y)dy\\
&=\int_{-s_-(T)}^{s_+(T)}f_n(y)\big(\nu-\mu\big)(dy)+\frac{1}{2}\int_{-s_-(T)}^{s_+(T)}f_n(y)U_{xx}(0,y)dy\nonumber
\end{align}
where the last expression follows from \eqref{eq:pdeU01}. 

To simplify the notation we set 
\begin{align*}
\Delta^-_T:=U_x(0,-s_-(T)+)\quad\text{and}\quad\Delta^+_T:=U_x(0,s_+(T)-)
\end{align*}
and notice that $\Delta^\pm_T$ may be non zero due to the lack of smooth-fit a the boundaries.
Now integrating by parts the last term on the right-hand side of \eqref{eq:alt01}, using \eqref{eq:pdeU02}, and the fact that $f_n(x)=f_n'(x)=f_n''(x)=0$ for $|x|\ge n+1$, we get
\begin{align}\label{eq:alt02}
\EE f_n(W^\nu_{\sigma_*\wedge T}) =&\EE \mathds{1}_{\{\sigma_*< T\}}f_n(W^\nu_{\sigma_*})-\int_{-s_-(T)}^{s_+(T)}f_n(y)\mu(dy)\nonumber\\
&+\int_\RR f_n(y)\nu(dy)+\frac{1}{2}\int_{-s_-(T)}^{s_+(T)}f_n''(y)U(0,y)dy\\
&+\frac{1}{2}\left[f_n(s_+(T))\Delta^+_T\mathds{1}_{\{s_+(T)\le n+1\}}-f_n(-s_-(T))\Delta^-_T\mathds{1}_{\{s_-(T)\le n+1\}}\right].\nonumber
\end{align}
Direct comparison of \eqref{eq:alt02} and \eqref{eq:intbypart00} then gives for all $n\ge 1$
\begin{align*}
\EE \mathds{1}_{\{\sigma_*< T\}}f_n(W^\nu_{\sigma_*})=&\int_{-s_-(T)}^{s_+(T)}f_n(y)\mu(dy)\nonumber\\
&-\frac{1}{2}\left[f_n(s_+(T))\Delta^+_T\mathds{1}_{\{s_+(T)\le n+1\}}-f_n(-s_-(T))\Delta^-_T\mathds{1}_{\{s_-(T)\le n+1\}}\right].
\end{align*}
Taking limits as $n\to\infty$ and using dominated convergence and pointwise convergence we have
\begin{align}\label{eq:skor0}
\EE \mathds{1}_{\{\sigma_*< T\}}f(W^\nu_{\sigma_*})=&\int_{-s_-(T)}^{s_+(T)}f(y)\mu(dy)\nonumber\\
&-\frac{1}{2}\left[f(s_+(T))\Delta^+_T\mathds{1}_{\{s_+(T)<+\infty\}}-f(-s_-(T))\Delta^-_T\mathds{1}_{\{s_-(T)<+\infty\}}\right].
\end{align}

It remains to take limits as $T\to\infty$.
If there exists $t_0>0$ such that $s_+(t_0)=s_-(t_0)=+\infty$, then the proof is complete because $s_+(t)=s_-(t)=+\infty$ for all $t\ge t_0$ and we only need to take $T\ge t_0$ in the last expression above. As it will be clarified in Corollary \ref{cor:finite} this situation never occurs in practice.

Let us now analyse the case in which there exists $t_0>0$ such that $s_+(t_0)=+\infty$ whereas $s_-(t)<+\infty$ for all $t>0$. The remaining cases, with $s_+(t)<+\infty$ for all $t>0$ and $s_-(t)\le +\infty$, may be addressed by the same methods. 
\vspace{+4pt}

\noindent\emph{Case 1.} [$\mu_-=+\infty$].\\
In this case \eqref{eq:b1} implies $s_-(T)\to\infty$ as $T\to\infty$ with $|s_-(T)|<+\infty$ for all $T>0$, and Corollary \ref{cor:smf2} implies $\Delta^-_T\to 0$. Hence taking limits as $T\to\infty$, using dominated convergence and \eqref{eq:skor0} we get
\begin{align}\label{eq:skor1}
\EE \mathds{1}_{\{\sigma_*<\infty\}}f(W^\nu_{\sigma_*})
=\int_\RR f(y)\mu(dy).
\end{align}
\vspace{+4pt}

\noindent\emph{Case 2.} [$\mu_-<+\infty$ and $\mu(\{-\mu_-\})=0$].\\
In this case $G'$ is continuous at $-\mu_-$, therefore \eqref{Vx-bound01} implies 
$\Delta^-_T\to 0$ as $T\to\infty$ since $s_-(T)\to\mu_-$. Hence arguing as in case 1 above we get \eqref{eq:skor1}.
\vspace{+4pt}

\noindent\emph{Case 3.} [$\mu_-<+\infty$ and $\mu(\{-\mu_-\})>0$].\\
This case requires more work. We approximate the measure $\mu$ via a sequence of measures $(\mu_k)_k$ whose cumulative distributions are constructed as follows: for each $k\ge 1$
\begin{align}\label{muk0}
F_{\mu_k}(x)=\left\{
\begin{array}{ll}
0, & x<-\mu_--\tfrac{1}{k}\\[+4pt]
k\cdot(x+\mu_-+\tfrac{1}{k})F_{\mu}(-\mu_-), & x\in[-\mu_--\tfrac{1}{k},-\mu_-)\\[+4pt]
F_\mu(x), & x\in[-\mu_-,+\infty).
\end{array}
\right.
\end{align}
Since $F_{\mu_k}(x)\to F_\mu$ as $k\to\infty$ for all points $x$ where $F_\mu$ is continuous, then $\mu_k\rightharpoonup\mu$ (see \cite{Sh}, Thm.~1, Ch.~3.1). It is important to notice that $F_{\mu_k}$ is continuous at the lower endpoint of its support, i.e.~at $-\mu^{(k)}_-:=-\mu_--1/k$. 

Letting $G_k$ be defined as in \eqref{eq:G} but with $F_\mu$ replaced by $F_{\mu_k}$ we can now consider the corresponding problem \eqref{eq:V} with value function denoted by $V_k$. Repeating the characterisation of the optimal stopping region for this problem we obtain the relative optimal boundaries $b^{(k)}_\pm$, which then produce two time-reversed boundaries $s^{(k)}_\pm$. In particular it is not hard to verify that \eqref{eq:b1} in this case implies that $\lim_{T\to\infty}s^{(k)}_-(T)= \mu_-+1/k$ and $\lim_{T\to\infty}s^{(k)}_+(T)= +\infty$ (for all $k$ sufficiently large). 

Since $F_{\mu_k}$ is continuous at $-\mu_-^{(k)}$ we argue as in case 2 above to get
\begin{align}\label{eq:skor-k0}
\EE \mathds{1}_{\{\sigma^{(k)}_*<\infty\}}f(W^\nu_{\sigma^{(k)}_*})=\int_\RR f(y)\mu_k(dy).
\end{align}
We claim here and prove in appendix that 
\begin{align}\label{limsk}
\lim_{k\to+\infty}\sigma^{(k)}_*=\sigma_*,\quad\text{$\PP$-a.s.}
\end{align}
so that taking limits in \eqref{eq:skor-k0}, again we obtain \eqref{eq:skor1}.
\vspace{+6pt}

Since \eqref{eq:skor1} holds for any $f\in C^2_b(\RR)$ we can extend to arbitrary continuous functions by a simple density argument. For any $f\in C_b(\RR)$ we consider an approximating sequence $(f_k)_{k\in\NN}\subset C^2_b(\RR)$ such that $f_k\to f$ pointwise as $k\to\infty$. For each $f_k$ the equation \eqref{eq:skor1} holds, then taking limits as $k\to\infty$ and using dominated convergence we obtain \eqref{eq:skb}.
\end{proof}

As corollaries of the above result we obtain interesting and non trivial regularity properties for the free-boundaries of problem \eqref{eq:V}. These are fine properties which are difficult to obtain in general via a direct probabilistic study of the optimal stopping problem. Namely we obtain: $(i)$ flat portions of either of the two boundaries may occur if and only if $\mu$ has an atom at the corresponding point (i.e.~$G_t+\tfrac{1}{2}G_{xx}$ has an atom. See Corollary \ref{cor:atoms}); $(ii)$ jumps of the boundaries may occur if and only if $F_\mu$ is flat on an interval (see \eqref{eq;jumpb}, \eqref{eq;jumpb2} and Corollary \ref{cor:jumps}). Note that the latter condition corresponds to saying that $G_t+\tfrac{1}{2}G_{xx}=0$ on an interval is a necessary and sufficient condition for a jump of the boundary (precisely of the size of the interval) and therefore it improves results in \cite{DeA15} where only necessity was proven. It should also be noticed that Cox and Peskir \cite{Cox-Pe13} proved $(i)$ and $(ii)$ constructively but did not discuss its implications for optimal stopping problems.
\begin{coroll}\label{cor:atoms}
Let $x_0\in\RR$ be such that $\mu(\{x_0\})>0$ then
\begin{itemize}
\item[ i)] if $x_0>0$ there exist $0\le t_1(x_0)<t_2(x_0)<+\infty$ such that $s_+(t)=x_0$ for $t\in(t_1,t_2]$,
\item[ii)] if $x_0<0$ there exist $0\le t_1(x_0)<t_2(x_0)<+\infty$ such that $s_-(t)=x_0$ for $t\in(t_1,t_2]$.
\end{itemize}
On the other hand, let either $s_+$ or $s_-$ be constant and equal to $x_0\in\RR$ on an interval $(t_1,t_2]$, then $\mu(\{x_0\})>0$.
\end{coroll}
\begin{proof}
We prove $i)$ arguing by contradiction. First notice that if $x_0>0$ and $\mu(\{x_0\})>0$, then the upper boundary must reach $x_0$ for some $t_0>0$ due to Theorem \ref{thm:sk}. Let us assume that $s_+(t_0)=x_0$ for some $t_0>0$ and let us assume that $s_+$ is strictly increasing on $(t_0-\eps,t_0+\eps)$ for some $\eps>0$. Then $\mu(\{x_0\})=\PP(W^\nu_{\sigma_*}=x_0)=\PP(W^\nu_{t_0}=s_+(t_0))=0$, hence a contradiction.

To prove the final claim let us assume with no loss of generality $s_+(t)=x_0$ for $t\in(t_1,t_2]$, then $\mu(\{x_0\})=\PP(W^\nu_{\sigma_*}=x_0)=\PP(W^\nu_t=x_0\:\text{for some $t\in(t_1,t_2]$},\,\sigma_*>t_1)>0$.
\end{proof}

\begin{coroll}\label{cor:jumps}
Let $(a,b)\subset\RR$ be an open interval such that $\mu((a,b))=0$ and for any $\eps>0$ it holds $\mu((a,b+\eps))>0$, $\mu((a-\eps,b))>0$, i.e.~$a$ and $b$ are endpoints of a flat part of $F_\mu$. Then
\begin{itemize}
\item[ i)] If $s_+(t)=a$ for some $t>0$ then $s_+(t+)=b$;
\item[ii)] If $-s_-(t)=b$ for some $t>0$ then $-s_-(t+)=a$.
\end{itemize}
\end{coroll}
\begin{proof}
It is sufficient to prove $i)$ since the argument is the same for $ii)$. Let us assume $s_+(t+)<b$, then there exists $t'>t$ such that $s_+(u)< b$ for $u\in(t,t')$. With no loss of generality we also assume $s_+$ strictly monotone on $(t,t')$ otherwise $\mu$ should have an atom on $(s_+(t),s_+(t'))$ (see Corollary \ref{cor:atoms}) hence contradicting that $\mu((a,b))=0$. Then we have 
\begin{align*}
\mu((a,b))\ge\mu \big(\big(s_+(t+),s_+(t')\big)\big)=&\PP\big(W^\nu_{\sigma_*}\in\big(s_+(t+),s_+(t')\big)\big)\\[+3pt]
\ge&\PP\big(\sup_{t\le s\le t'}W^\nu_s\ge s_+(t'),\,\sigma_*>t\big)>0,
\end{align*}
which contradicts the assumptions.
\end{proof}

Notice that for $f\equiv 1$ \eqref{eq:skb} gives $\PP(\sigma_*<+\infty)=\mu(\RR)=1$. As anticipated in the proof of Theorem \ref{thm:sk}, this implies that there cannot exist a time $t_0>0$ such that $s_+(t)=s_-(t)=+\infty$ for all $t\ge t_0$. 
\begin{coroll}\label{cor:finite}
For all $t>0$, either $s_+(t)<+\infty$ or $s_-(t)<+\infty$ or both.
\end{coroll}

We conclude the paper with a discussion on the role of Assumption (D.1).

\begin{remark}\label{rem:ncvx}
As anticipated in Section \ref{sec:sett}, although Assumption (D.1) is not necessary to implement the methods illustrated in this paper, it is a convenient one for the clarity of exposition. Here we illustrate how our methods may be used to deal with a pair $\nu$ and $\mu$ which does not meet (D.1).

Take
\begin{align}
\nu(dx)=\frac{1}{2}\left(\delta_{-1}(x)+\delta_{1}(x)\right)dx,\quad \mu(dx)=\mathds{1}_{\left[-\tfrac{1}{2},\tfrac{1}{2}\right]}(x)dx.
\end{align}  
Then $G$ is non positive, it equals $-3/4$ on $(-\infty,-1)\cup(1,+\infty)$, it is increasing on $(-1,0)$ and decreasing on $(0,1)$, with maximum value $G(0)=0$. Arguing as in Proposition \ref{thm:OSinfiniteH}, for $T=+\infty$ we obtain $v(x)=0$ and $\CC_\infty=\RR\setminus\{0\}$. 

For $T<+\infty$, using the same arguments as in Section \ref{sec:optSt} one finds a non-connected continuation set of the form 
\begin{align}
\CC_T=\{(t,x)\in[0,T)\times\RR\,:\,x\in(-\infty,-b^T_-(t))\cup(b^T_+(t),+\infty)\}
\end{align}
where the functions $b^T_\pm$ are continuous on $[0,T)$, increasing and positive, with $b^T_\pm(T-)=\tfrac{1}{2}$.
Since $G'$ is continuous on $[-\tfrac{1}{2},\tfrac{1}{2}]$ we also have $V\in C^1([0,T)\times\RR)$ by the same arguments as those used in Section \ref{sec:C1}.

In the same spirit of Definition \ref{def:spm} we define $s_\pm$, continuous and decreasing, as the time reversal of $b_\pm^T$ for $T>0$. Notice that $s_\pm(t)\ge 0$ for all $t\ge0$ and $s_\pm(+\infty)=0$. Following Section \ref{sec:sk} we have
\begin{align}
\CC_\infty^-=\{(t,x)\in[0,+\infty)\times\RR\,:\,x\in(-\infty,-s_-(t))\cup(s_+(t),+\infty)\}.
\end{align}
Due to the fact that $\CC_T$ is not connected and $b^T_\pm\ge 0$, then for $(t,x)\in \CC_T$ the time-space Brownian motion $(t+s,x+B_s)_{s\ge 0}$ can only enter the stopping set $\DD_T$, by crossing $b^T_+$ if $x> 0$, and by crossing $-b^T_-$ if $x<0$. 

Proposition \ref{prop:Ut} holds in the same form and its proof can be repeated up to minor changes. In particular \eqref{eq:Ut-trans} reads
\begin{align}
-2 U^T_t(t,x)=p^{\CC}_-(0,-1,T-t,x)+p^{\CC}_-(0,1,T-t,x),
\end{align}
where indeed we notice that $p^{\CC}_-(0,-1,T-t,x)=0$ for $x>0$ and $p^{\CC}_-(0,1,T-t,x)=0$ for $x<0$, because $\CC^-_\infty$ is not connected.
Using the latter representation one can repeat step by step the arguments of proof of Theorem \ref{thm:sk}, with obvious changes, to obtain that 
\eqref{eq:sk} holds with 
$$\sigma_*:=\inf\{t>0\,:\,W^\nu_t\in[-s_-(t),s_+(t)]\}.$$
\end{remark}


\appendix

\section{Appendix}
\renewcommand{\theequation}{A-\arabic{equation}}

\begin{proof}[Proof of eq.~\eqref{eq:pde01}--\eqref{eq:pde03}]\label{app:freeb}

Condition \eqref{eq:pde02} and \eqref{eq:pde03} are obvious whereas to prove \eqref{eq:pde01} we use a well known argument (see for instance \cite[Sec.~7.1]{Pes-Shir}).
Since $\CC_T$ is an open set and it is not empty (see step 2 in the proof of Theorem \ref{thm:OS-b}) we can consider an open, bounded rectangular domain $\mathcal{U}\subset \CC_T$ with parabolic boundary $\partial_P\mathcal{U}$. Then the following boundary value problem
\begin{align}\label{Cauchy2}
u_t+\tfrac{1}{2}u_{xx}=0\quad\text{on $\mathcal{U}$ with $u=V$ on $\partial_P\mathcal{U}$}
\end{align}
admits a unique classical solution $u\in C^{1,2}(\mathcal{U})\cap C(\overline{\mathcal{U}})$ (cf.~for instance \cite[Thm.~9, Sec.~4, Ch.~3]{Fri08}). Fix $(t,x)\in\mathcal{U}$ and denote by $\tau_{\mathcal{U}}$ the first exit time of $(t+s,x+B_s)_{s\ge 0}$ from $\mathcal{U}$. Then Dynkin's formula gives
\begin{align*}
u(t,x)=\EE\left[ u(t+ \tau_{\mathcal{U}},x+B_{\tau_{\mathcal{U}}})\right]=\EE\left[ V(t+ \tau_{\mathcal{U}},x+B_{\tau_{\mathcal{U}}})\right]=V(t,x)
\end{align*}
where the last equality follows from the fact that $V(t+ s\wedge\tau_*,x+B_{s\wedge\tau_*})$, $s\ge 0$ is a martingale according to standard optimal stopping theory and $\tau_{\mathcal{U}}\le \tau_*$, $\PP$-a.s. 

Since $\mathcal{U}$ is arbitrary in $\CC_T$ the equation \eqref{eq:pde01} follows.
\end{proof}

\begin{proof}[Proof of Lemma \ref{lem:contST}]
Because of \eqref{st-interior} we have $\tilde \tau_*(t,x)=0$, $\PP$-a.s. In particular this means that for any fixed $\omega\in\Omega\setminus\mathcal N$, with $\mathcal N$ a null set, and for any $\delta>0$ there is $s=s(\omega)\in(0,\delta)$ such that $(t+s,x+B_s(\omega))\in\DD^\circ_T$. Since $(t_n+s,x_n+B_s(\omega))\to (t+s,x+B_s(\omega))$ as $n\to\infty$, and $\DD^\circ_T$ is open, then there exists $N_\omega\in\mathbb{N}$ such that $(t_n+s,x_n+B_s(\omega))\in\DD^\circ_T$ for all $n\ge N_\omega$. Thus $\tilde \tau_*(t_n,x_n)(\omega)<\delta$ for all $n\ge N_\omega$ and 
\begin{align*}
\limsup_{n\to\infty}\tilde \tau_*(t_n,x_n)(\omega)<\delta.
\end{align*}
Recalling \eqref{st-interior} and that $\delta$ was arbitrary we obtain
\begin{align*}
\lim_{n\to\infty} \tau_*(t_n,x_n)(\omega)=\lim_{n\to\infty}\tilde \tau_*(t_n,x_n)(\omega)=0.
\end{align*}
Since $\omega$ was also arbitrary we conclude the proof.
\end{proof}

\begin{proof}[Proof of Lemma \ref{cor:contST-t}]
For simplicity set $\tau_*=\tau_*(t,x)$ and $\tau_h=\tau_*(t_h,x)$. By monotonicity of the optimal boundaries it is not hard to see that $(\tau_h)_{h\ge0}$ forms a family which is decreasing in $h$ with $\tau_h\ge\tau_*$ for all $h$, $\PP$-a.s. We denote $\tau_\infty:=\lim_{h\to\infty}\tau_h$, $\PP$-a.s., so that $\tau_\infty\ge\tau_*$ and arguing by contradiction we assume that there exists $\Omega_0\subset\Omega$ such that $\PP(\Omega_0)>0$ and $\tau_\infty-\tau_*>0$ on $\Omega_0$. Notice that $\tau_*<T-t$ on $\Omega_0$, otherwise $\tau_\infty>\tau_*$ leads immediately to a contradiction. 

Let us pick $\omega\in\Omega_0$ and with no loss of generality let us assume that 
\begin{align}\label{eq:up0}
x+B_{\tau_*}(\omega)\ge b_+(t+\tau_*(\omega))
\end{align}
(similar arguments hold for $b_-$). Since we are on $\Omega_0$, then there exists $\delta_\omega>0$ such that $\tau_\infty(\omega)-\tau_*(\omega)\ge\delta_\omega$ and for all $h>0$ it must be
\begin{align}\label{th1}
x+B_{\tau_*+s}(\omega)< b_+(t_h+\tau_*(\omega)+s),\qquad s\in(0,\delta_\omega/2].
\end{align}
For any $\eps\in(0,\delta_\omega/2)$ we find $h_\eps$ sufficiently large to get $t-t_h<\eps$ for $h\ge h_\eps$ and consequently $t_h+s\ge t$ for $s\in(\eps,\delta_\omega/2]$. Monotonicity of $b_+$ implies that for $h\ge h_\eps$ we have
\begin{align*}
b_+(t_h+\tau_*(\omega)+s)\le b_+(t+\tau_*(\omega)) \qquad s\in(\eps,\delta_\omega/2]
\end{align*}
and hence, by \eqref{th1}, also 
\begin{align}\label{th2}
x+B_{\tau_*+s}(\omega)< b_+(t+\tau_*(\omega)),\qquad s\in(\eps,\delta_\omega/2].
\end{align}
Letting now $\eps\to 0$ in \eqref{th2}, the latter and \eqref{eq:up0} would imply
$B_{\tau_*+s}(\omega)-B_{\tau_*}(\omega)\le 0$ for $s\in (0,\delta_\omega/2]$, which contradicts the law of iterated logarithm.
\end{proof}

\begin{proof}[Proof of Proposition \ref{prop:sm-fit}]
We only provide a full proof for \eqref{Vx-bound00} as the argument for \eqref{Vx-bound01} is completely analogous up to trivial changes. 
Let $t\in[0,T)$ and $x:=b_+(t)<+\infty$ then it is easy to see that
\begin{align}\label{smf-b00}
\limsup_{\eps\to0}\frac{1}{\eps}\left(V(t,x)-V(t,x-\eps)\right)\le \limsup_{\eps\to 0}\frac{1}{\eps}\left(G(x)-G(x-\eps)\right)=G'(x-).
\end{align} 
Moreover \eqref{eq:pde01} implies $V_{xx}=-2V_t\ge 0$ in $\CC_T$ so that $V_x(t,\,\cdot\,)$ is increasing for all $x\in(-b_-(t),b_+(t))$ and its limit at $x=b_+(t)$ is well defined. Hence \eqref{smf-b00} implies
\begin{align}\label{smf-b01}
V_x(t,x-)\le G'(x-).
\end{align}
For the other inequality in \eqref{Vx-bound00} we denote 
\begin{align*}
\tau_\eps:=&\inf\{s\in[0,T-t]\,:\,(t+s, B^{x-\eps}_s)\in\DD_T\},\\[+3pt]
\tau_{a_-}:=&\inf\{s\ge 0\,:\,B^{x-\eps}_s\le -a_-\},
\end{align*}
set $\rho_\eps:=\tau_\eps\wedge\tau_{a_-}$, and recall that 
\begin{align*}
Y^\eps_s:=V(t+s\wedge\rho_\eps,B^{x-\eps}_{s\wedge\rho_\eps})\quad\text{is a martingale},
\end{align*}
whereas $Y_s:=V(t+s,B^{x}_s)$ is a supermartingale for $s\in[0,T-t]$. We notice that 
\begin{align}\label{tau=1}
\PP(\tau_{a_-}>0)=1.
\end{align}
If $-a_-<a_+$ the result is trivial. If $a_-=a_+=0$, then $\nu(\{0\})=1$ and $b_+(t)>0$ for $t\in[0,T)$ by (iv) in Theorem \ref{thm:OS-b}. Hence $b_+(t)-\eps>0$ for $\eps$ sufficiently small and \eqref{tau=1} holds. 

Using the (super)martingale property of $Y$ and $Y^\eps$ we have
\begin{align}\label{smf-b03}
V(t,x)-V(t,x-\eps)\ge&\, \EE\left[V(t+\rho_\eps,B^x_{\rho_\eps})-V(t+\rho_\eps,B^{x-\eps}_{\rho_\eps})\right]\\[+4pt]
=&\,\EE\left[\mathds{1}_{\{\tau_\eps<\tau_{a_-}\}\cap\{\rho_\eps\le\delta\}}\left(G(B^x_{\tau_\eps})-G(B^{x-\eps}_{\tau_\eps})\right)\right]\nonumber\\[+4pt]
&+\EE\left[\mathds{1}_{\{\tau_\eps>\tau_{a_-}\}\cap\{\rho_\eps\le\delta\}}\left(V(t+\tau_{a_-},B^x_{\tau_{a_-}})-V(t+\tau_{a_-},B^{x-\eps}_{\tau_{a_-}})\right)\right]\nonumber\\[+4pt]
&+\EE\left[\mathds{1}_{\{\rho_\eps>\delta\}}\left(V(t+\rho_\eps,B^x_{\rho_\eps})-V(t+\rho_\eps,B^{x-\eps}_{\rho_\eps})\right)\right].\nonumber
\end{align}
Recalling the Lipschitz continuity of $V(t,\cdot)$ (Proposition \ref{prop:V}) and since $B^x_{\rho}-B^{x-\eps}_{\rho}=\eps$ $\PP$-a.s.~for any stopping time $\rho$, we obtain the lower bounds
\begin{align*}
&\EE\left[\mathds{1}_{\{\tau_\eps>\tau_{a_-}\}\cap\{\rho_\eps\le\delta\}}\left(V(t+\tau_{a_-},B^x_{\tau_{a_-}})-V(t+\tau_{a_-},B^{x-\eps}_{\tau_{a_-}})\right)\right]\ge-\eps\,L_G\PP(\tau_\eps>\tau_{a_-},\,\rho_\eps \le \delta), \nonumber\\[+4pt]
&\EE\left[\mathds{1}_{\{\rho_\eps>\delta\}}\left(V(t+\rho_\eps,B^x_{\rho_\eps})-V(t+\rho_\eps,B^{x-\eps}_{\rho_\eps})\right)\right]\ge -\eps\, L_G\PP(\rho_\eps>\delta).
\end{align*}
We notice that since $b_+$ is decreasing, then on the event $\{\rho_\eps\le \delta\}\cap\{\tau_\eps<\tau_{a_-}\}$ one has $x-\eps+B_{\tau_\eps}\ge b_+(t+\delta)\ge a_+$. Moreover $G$ is concave and increasing on $[a_+,+\infty)$ and therefore also on the interval $(B^{x-\eps}_{\tau_\eps},B^x_{\tau_\eps})$ when considering the event $\{\rho_\eps\le \delta\}\cap\{\tau_\eps<\tau_{a_-}\}$. Using these facts we obtain
\begin{align*}
&\EE\left[\mathds{1}_{\{\tau_\eps<\tau_{a_-}\}\cap\{\rho_\eps\le\delta\}}\left(G(B^x_{\tau_\eps})-G(B^{x-\eps}_{\tau_\eps})\right)\right]\\[+4pt]
&\ge \eps\,\EE\left[\mathds{1}_{\{\tau_\eps<\tau_{a_-}\}\cap\{\rho_\eps\le\delta\}}G'(x+B_{\tau_\eps})\right]\ge \eps\,G'(b_+(t)+\eps)\PP(\tau_\eps<\tau_{a_-},\rho_\eps\le\delta)
\end{align*}
where for the last inequality we have used again concavity of $G$ and that $x-\eps+B_{\tau_\eps}\le b_+(t)$ because the boundary $b_+$ is monotonic decreasing.

Plugging in \eqref{smf-b03} the lower bounds obtained for the terms on the right-hand side, and dividing by $\eps$ we find 
\begin{align*}
\frac{1}{\eps}(V(t,x)-V(t,x-\eps))\ge &\,G'(b_+(t)+\eps)\PP\left(\tau_\eps<\tau_{a_-},\,\rho_\eps\le\delta \right)\\[+4pt]
&-L_G\left(\PP(\tau_\eps>\tau_{a_-},\,\rho_\eps \le \delta)+\PP(\rho_\eps>\delta)\right).
\end{align*}
Taking liminf as $\eps\to0$ and using that $\tau_\eps\to 0$ $\PP$-a.s.~due to Lemma \ref{lem:contST} we immediately obtain
\begin{align*}
V_x(t,x-)\ge G'(b_+(t)+) = G'(x).
\end{align*} 
The latter and \eqref{smf-b01} prove \eqref{Vx-bound00}. 
\end{proof}

\begin{proof}[Proof of Lemma \ref{lem:sm-fit2}]
We will only give details for the limits involving $b_+$ as those involving $b_-$ can be obtained in the same way. 
\vspace{+4pt}

\noindent\emph{Step 1 (Proof of (ii)).} If $\mu(\{\hat{b}_+\})=\nu(\{\hat b_+\})=0$ then $G'$ is continuous at $\hat{b}_+$. Moreover since $b_+(t)\to \hat{b}_+$ as $t\to T$ we can take limits as $t\to T$ in \eqref{Vx-bound00} and obtain \eqref{Vx-T00}. If instead $\nu(\{\hat b_+\})>\mu(\{\hat{b}_+\})=0$, i.e.~$a_+=\hat b_+$ and $\nu$ has an atom at that point, then (iv) of Theorem \ref{thm:OS-b} implies that $b_+(t)$ converges to $a_+$, as $t\to T$, strictly from above. Hence, by right-continuity of $G'$ we get 
\begin{align*}
\lim_{t\to T}G'(b_+(t))=\lim_{t\to T}G'(b_+(t)-)=G'(\hat b_+)
\end{align*}
and (ii) holds due to \eqref{Vx-bound00}.
\vspace{+4pt}

\noindent\emph{Step 2 (Proof of (i)).} The more interesting case is when $\mu(\{\hat{b}_+\})>0$ and therefore $\hat b_+>a_+$ due to Assumption (D.2). For this part of the proof it is convenient to use the notation $\EE_{t,x}[\,\cdot\,]=\EE[\,\cdot\,|B_t=x]$ and to think of $\Omega$ as the space of continuous functions, with $\theta_\cdot:\Omega\to\Omega$ denoting the shifting operator. 

In particular we take $t\in [t_+,T)$ so that $b_+(t)=\hat{b}_+$ and $V(t,b_+(t))=G(b_+(t))$ (see Lemma \ref{lem:flatb}). We also pick $a\in(a_+,\hat{b}_+)$ and denote $\tau_a:=\inf\{s\ge0\,:\,X_s\le a\}$. 
For $\eps>0$ such that $\hat{b}_+-\eps>a$ we have
\begin{align}\label{up0}
V(t,\hat{b}_+)-V(t,\hat{b}_+-\eps)=G(\hat{b}_+)-G(\hat{b}_+-\eps)-\int_{\RR}\EE_{t,\hat{b}_+-\eps}\big[L^z_{\tau_*}\big](\nu-\mu)(dz)
\end{align}
with $\tau_*$ as in \eqref{eq:optst}. To find a lower bound for the last term in \eqref{up0} we notice that $L^z_{\tau_*}\mathds{1}_{\{\tau_*\le\tau_a\}}\nu(dz)=0$ and $L^z_{\tau_a}\nu(dz)=0$, $\PP_{t,\hat{b}_+-\eps}$-a.s.~and use the strong Markov property as follows.
\begin{align*}
\int_{\RR}\EE_{t,\hat{b}_+-\eps}\big[L^z_{\tau_*}\big](\nu-\mu)(dz)\le&\, \int_{\RR}\EE_{t,\hat{b}_+-\eps}\big[\mathds{1}_{\{\tau_*>\tau_a\}}L^z_{\tau_*}\big]\nu(dz)\\[+4pt]
=&\, \int_{\RR}\EE_{t,\hat{b}_+-\eps}\big[\mathds{1}_{\{\tau_*>\tau_a\}}\left(L^z_{\tau_a}+\EE_{t,\hat{b}_+-\eps}\big[L^z_{\tau_*}\circ\theta_{\tau_a}\big|\cF_{\tau_a}\big]\right)\big]\nu(dz)\nonumber\\[+4pt]
=&\, \int_{\RR}\EE_{t,\hat{b}_+-\eps}\big[\mathds{1}_{\{\tau_*>\tau_a\}}\EE_{\tau_a,a}\big[L^z_{\tau_*}\big]\big]\nu(dz)\nonumber\\[+4pt]
\le&\, \PP_{t,\hat{b}_+-\eps}(\tau_a<\tau_*)\int_{\RR}\sup_{t\le s\le T}\EE_{s,a}\big[L^z_{\tau_*}\big]\nu(dz)\nonumber.
\end{align*}
Setting $g(t):=\int_{\RR}\sup_{t\le s\le T}\EE_{s,a}\big[L^z_{\tau_*}\big]\nu(dz)$ and substituting the above bound in \eqref{up0} we get
\begin{align}\label{up3}
V(t,\hat{b}_+)-V(t,\hat{b}_+-\eps)\ge G(\hat{b}_+)-G(\hat{b}_+-\eps)-g(t)\PP_{t,\hat{b}_+-\eps}(\tau_a<\tau_*).
\end{align}
Notice that since $b_+(t)=\hat{b}_+$ for all $t\in[t_+,T]$ then $\{\tau_a<\tau_*\}\subset\{\tau_a<\tau_{\hat{b}_+}\wedge(T-t)\}$, $\PP_{t,\hat{b}_+-\eps}$-a.s.~where $\tau_{\hat{b}_+}:=\inf\{s \ge 0\,:\,X_s\ge \hat{b}_+\}$. Therefore 
\begin{align*}
\PP_{t,\hat{b}_+-\eps}(\tau_a<\tau_*)\le\PP_{t,\hat{b}_+-\eps}(\tau_a<\tau_{\hat{b}_+}\wedge(T-t))\le\PP_{t,\hat{b}_+-\eps}(\tau_a<\tau_{\hat{b}_+})=\frac{\eps}{\hat{b}_+-a}
\end{align*}
where for the last equality follows by well known properties of the scale function of Brownian motion. Plugging the above in \eqref{up3}, dividing by $\eps$ and taking limits as $\eps\to0$ gives
\begin{align}
V_x(t,\hat{b}_+-)\ge G'(\hat{b}_+-)-g(t)(\hat{b}_+-a)^{-1}.
\end{align}
Now letting $t\to T$ and noticing that $g(t)\to0$ we obtain \eqref{Vx-T00} upon recalling \eqref{Vx-bound00}.
\end{proof}

\begin{proof}[Proof of Lemma \ref{lem:Vx-asympt}]
We only prove the statement for $\supp\{\mu\}\cap\RR_+=\emptyset$ as the arguments for the the other case are the same. Let $t\in[0,T]$ and $x>0$, so that $(t,x)\in\CC_T$ and $(t,x+\eps)\in\CC_T$ for all $\eps>0$, since the stopping set is all contained in $[0,T]\times\RR_-$ (recall $(ii)$ of Theorem \ref{thm:OS-b}).

For $\tau_*=\tau_*(t,x)$ we have  
\begin{align*}
\frac{1}{\eps}\left(V(t,x+\eps)-V(t,x)\right)\ge \frac{1}{\eps}\EE\Big[G(x+\eps+B_{\tau_*})-G(x+B_{\tau_*})\Big],
\end{align*}
and
\begin{align*}
\frac{1}{\eps}\left(V(t,x)-V(t,x-\eps)\right)\le \frac{1}{\eps}\EE\Big[G(x+B_{\tau_*})-G(x-\eps+B_{\tau_*})\Big].
\end{align*}
Since $V\in C^{1,2}$ inside $\CC_T$ and $G'$ is right-continuous then taking limits as $\eps\to0$ gives 
\begin{align}\label{bb}
\EE_x G'(B_{\tau_*})\le V_x(t,x)\le \EE_x G'(B_{\tau_*}-).
\end{align}
Notice that $G'(x)\to 0$ as $x\to\infty$ (recall that $\nu(\{+\infty\})=0$), hence for any $\eps>0$ there exists $x_\eps>0$ such that $|G'(x)|\le \eps$ for $x\in[x_\eps,+\infty)$. We fix $\eps>0$ and with no loss of generality consider $x>x_\eps$. Then we have
\begin{align*}
\EE_x\big| G'(B_{\tau_*})\big|=&\EE_x \big[\big|G'(B_{\tau_*})\big|\mathds{1}_{\{\tau_*<T-t\}}+\big|G'(B_{T-t})\big|\mathds{1}_{\{\tau_*=T-t\}}\big]\\[+4pt]
\le& L_G\PP_x(\tau_*<T-t)+\EE_x\big[\big|G'(B_{T-t})\big|\mathds{1}_{\{\tau_*=T-t\}\cap\{B_{T-t}\le x_\eps\}}\big]+\eps\\[+4pt]
\le& L_G\Big(\PP_x(\tau_*<T-t)+\PP_x(B_{T-t}\le x_\eps)\Big)+\eps.
\end{align*}
An analogous inequality clearly holds for $\EE_x\big| G'(B_{\tau_*}-)\big|$. 

Since $x>x_\eps$, then both 
$\PP_x(\tau_*<T-t)$ and $\PP_x(B_{T-t}\le x_\eps)$ are bounded from above by $\PP(\sup_{0\le s\le T}\,\big|B_s\big|\ge |x_\eps-x|)$. Therefore from \eqref{bb} and the estimates above we obtain
\begin{align*}
\sup_{0\le t\le T}\big|V_x(t,x)\big|\le 2L_G\PP\Big(\sup_{0\le s\le T}\,\big|B_s\big|\ge |x_\eps-x|\Big)+\eps.
\end{align*} 
Letting $x\to\infty$ and recalling that $\eps>0$ was arbitrary the proof is completed. 
\end{proof}

\begin{proof}[Proof of Lemma \ref{lemma:Hunt}]
The proof is a generalisation of the proof of \cite[Thm.~24.7]{Kal} and it will be sufficient to give it in the case with $t=0$ and $s=T$. In particular it is enough to show that for any $A,B\in\mathcal{B}(\RR)$ with $A\subseteq[-b_-(0),b_+(0)]$ and $B\subseteq[-s_-(0),s_+(0)]$ one has
\begin{align}\label{eq:hunt00}
\int_A \PP_x(B_T\in B\,,\,T\le\tau_*)dx=\int_B \PP_x(W_T\in A\,,\,T\le\tau_-)dx.
\end{align}
Recalling \eqref{st-interior} and \eqref{eq:equivtaus}, we find it convenient (with no loss of generality) to prove \eqref{eq:hunt00} with $\tilde \tau_*$ and $\tilde \tau_-$ instead of $\tau_*$ and $\tau_-$.

For the sake of this proof and with no loss of generality we can consider the canonical space $\Omega= C([0,\infty))$ with the Borel $\sigma$-algebra $\cF=\mathcal{B}\big(C([0,\infty))\big)$. Given that \eqref{eq:hunt00} only involves the laws of $B$ and $W$ we can simplify the notation and consider a single Brownian motion $X=(X_t)_{t\ge0}$ defined as the coordinate process $X_t(\omega)=\omega(t)$ with its filtration $(\cF^X_t)_{t\ge0}$ augmented with the $\PP$-null sets. With a slight abuse of notation, here we denote by $\PP$ the Wiener measure on $(\Omega,\cF)$. In this setting $\tilde \tau_*$ coincides with the first exit time of $(X_t)_{t\ge0}$ from $[-b_-(t),b_+(t)]$, $t\in[0,T]$ and $\tilde{\tau}_-$ coincides with the first (strictly positive) exit time of $(X_t)_{t\ge0}$ from $[-s_-(t),s_+(t)]$, $t\ge0$.

Due to \eqref{st-interior} and \eqref{eq:equivtaus2} it is not difficult to see that 
\begin{align}\label{eq:hunt01}
\{T\le\tilde{\tau}_*\}=\bigcap_{q\in[0,T]\cap\mathbb{Q}}\big\{X_q\in[-b_-(q),b_+(q)]\big\}.
\end{align}
and 
\begin{align}\label{eq:hunt03}
\{T\le\tilde{\tau}_-\}=\bigcap_{q\in[0,T]\cap\mathbb{Q}}\big\{X_q\in[-s_-(q),s_+(q)]\big\}.
\end{align}

For simplicity and without loss of generality we assume $T\in\mathbb{Q}$. Now, 
we can consider a sequence $(\pi_n)_{n\in\mathbb{N}}$ of dyadic partitions of $[0,T]$ defined by $\pi_n:=\{t^n_0,t^n_1,\ldots t^n_n\}$ where $t^n_k:=\tfrac{k}{2^n}T$, $k=1,2,\ldots 2^n$ and then
\begin{align}\label{eq:hunt04}
&\{T\le\tilde{\tau}_*\}=\lim_{n\to\infty}\bigcap_{q\in\pi_n}\big\{X_q\in[-b_-(q),b_+(q)]\big\},\\[+4pt]
&\{T\le\tilde{\tau}_-\}=\lim_{n\to\infty}\bigcap_{q\in\pi_n}\big\{X_q\in[-s_-(q),s_+(q)]\big\}.
\end{align}
We set $h_n=t^n_{k+1}-t^n_k=T/2^n$ and denote $p^n_h(x,y)=\tfrac{1}{\sqrt{2\pi h_n}}\exp-\tfrac{1}{2h_n}(x-y)^2$. By using monotone convergence and Chapman-Kolmogorov equation we obtain
\begin{align}\label{eq:hunt05}
&\int_B\PP_x(X_T\in A,\,T\le\tilde{\tau}_-)dx\\
&=\lim_{n\to\infty}\int_B\PP_x(X_q\in[-s_-(q),s_+(q)]\:\text{for all}\:q\in\pi_n,\, X_T\in A)dx\nonumber\\
&=\lim_{n\to\infty}\int p^n_h(x_0,x_1)p^n_h(x_1,x_2)\ldots p^n_h(x_{2^n-1},x_{2^n})dx_0\,dx_1\,\ldots dx_{2^n} \nonumber
\end{align}
where the last integral is taken with respect to $x_0\in B$, $x_{2^n}\in A$ and $x_k\in[-s_-(t^n_k),s_+(t^n_k)]$ for $k=1,2,\ldots 2^n-1$. We interchange order of integration, relabel variables $x_{2^n-k}=y_k$ for $k=0,1,2,\ldots 2^n$ and use symmetry of the heat kernel along with the fact that $s_\pm(q)=b_\pm(T-q)$ to conclude
\begin{align*}
&\int_B\PP_x(X_T\in A,\,T\le\tilde{\tau}_-)dx\\
&=\lim_{n\to\infty}\int p^n_h(y_0,y_1)p^n_h(y_1,y_2)\ldots p^n_h(y_{2^n-1},y_{2^n})dy_0\,dy_1\,\ldots dy_{2^n}\nonumber\\
&=\lim_{n\to\infty}\int_A\PP_x(X_q\in[-b_-(q),b_+(q)]\:\text{for all}\:q\in\pi_n,\, X_T\in B)dx\nonumber\\
&=\int_A\PP_x(X_T\in B,\,T\le\tilde{\tau}_*)dx.\nonumber
\end{align*}
Hence \eqref{eq:hunt00} follows and the generalisation to arbitrary $t<s$ can be obtained with the same arguments.
\end{proof}

\begin{proof}[Proof of \eqref{eq:convCd}] It is sufficient to show that $b^\delta_+(t)\downarrow b_+(t)$ for all $t\in[0,T)$ since the proof for $b_-$ is analogous and the convergence of the related sets easily follows from the same arguments. Note that for each $t$ the limit $b^0_+(t):=\lim_{\delta\to0}b^\delta_+(t)$ exists and $b^0_+(t)\ge b_+(t)$ since $\delta\mapsto b^\delta_+(t)$ decreases as $\delta\to0$ and $b^\delta_+(t)\ge b_+(t)$ for all $\delta>0$. Let us assume that there exists $\bar{t}\in[0,T)$ such that $b^0_+(\bar{t})>b_+(\bar{t})$. Pick $\bar{x}\in (b_+(\bar{t}),b^0_+(\bar{t}))$, then by definition of $b^\delta_+$ it should follow that $\inf_{\delta>0}V^\delta(\bar{t},\bar{x})-G^\delta(\bar{x})\ge \eta>0$ for some $\eta=\eta(\bar{t},\bar{x})$. However this is clearly impossible since $V^\delta(\bar{t},\bar{x})-G^\delta(\bar{x})$ converges to $V(\bar{t},\bar{x})-G(\bar{x})=0$ as $\delta \to 0$ by \eqref{eq:Vdelta01}.
\end{proof}

\begin{proof}[Proof of \eqref{eq:Vdelta01}] We denote $\|\,\cdot\,\|_\infty$ the $L^\infty(\RR)$ norm. By direct comparison we obtain
\begin{align}\label{eq:convVd00}
\big(V^\delta-V\big)(t,x)\le& \sup_{0\le\tau\le T-t}\EE_x2\int_0^{B_\tau}\big(F_\mu-F^\delta_\mu\big)(z)dz\\[+3pt]
=&2\big\|F_\mu-F^\delta_\mu\big\|_\infty\sup_{0\le\tau\le T-t}\EE_x\big|B_\tau\big|\nonumber
\end{align}
and the same bound can be found for $(V-V^\delta)(t,x)$. Then by an application of Jensen inequality and using that $\EE_x (B_\tau)^2=x^2+\EE_0B^2_\tau=x^2+\EE_0\tau$ we get
\begin{align}\label{eq:convVd01}
\big|V^\delta-V\big|(t,x)\le2\big\|F_\mu-F^\delta_\mu\big\|_\infty\sup_{0\le\tau\le T-t}\Big(\EE_x\big|B_\tau\big|^2\Big)^{\tfrac{1}{2}}\le 2(|x|+\sqrt{T})\big\|F_\mu-F^\delta_\mu\big\|_\infty.
\end{align}
The latter goes to zero as $\delta\to0$ by \eqref{eq:Fdelta01}, uniformly for $t\in[0,T]$ and $x$ in a compact.
\end{proof}

\begin{proof}[Proof of \eqref{eq:convTd}] Thanks to \eqref{eq:equivtaus} and \eqref{eq:equivtaus2} it is sufficient to prove that $\tilde \tau^\delta_-\downarrow \tilde \tau_-$ as $\delta\to0$. We denote $\tau_0:=\lim_{\delta\to0}\tilde{\tau}^\delta_-$, $\PP$-a.s.~(the limit exists since the sequence is monotone by \eqref{eq:convCd}). Note that $\tau_0\ge\tilde{\tau}_-$ and let us now prove that the reverse inequality also holds.

Fix $\hat{\omega}\in\Omega$, then if $\tilde{\tau}_-(\hat{\omega})=+\infty$ we immediately obtain $\tau_0(\hat{\omega})=\tilde{\tau}_-(\hat{\omega})$. On the other hand let $\eta_{\hat{\omega}}>0$ be such that $\tilde{\tau}_-(\hat{\omega})<\eta_{\hat{\omega}}$. Then there exists $t\in(\tilde{\tau}_-(\hat{\omega}),\eta_{\hat{\omega}})$ (also depending on $\hat{\omega}$) such that $W^\nu_t(\hat{\omega})\notin[-s_-(t),s_+(t)]$, i.e.~with no loss of generality we may assume that there exists $\eps_{t,\hat{\omega}}>0$ and such that $W^\nu_t(\hat{\omega})>s_+(t)+\eps_{t,\hat{\omega}}$. By \eqref{eq:convCd} it then follows that $W^\nu_t(\hat{\omega})>s^\delta_+(t)$ for all $\delta$ sufficiently small and hence $\tau_0(\hat{\omega})<\eta_{\hat{\omega}}$. Since $\eta_{\hat{\omega}}$ was arbitrary we conclude that $\tau_{0}(\hat{\omega})\le \tilde{\tau}(\hat{\omega})$. Repeating the argument for all $\omega\in\Omega$ the claim is proved. 
\end{proof}

\begin{proof}[Proof of a refined version of Lemmas \ref{lem:Vx-asympt} and \ref{lemma:Vt-meas}] Here we discuss a technicality needed to make the proof of $V^\delta_t\in C([0,T)\times\RR)$ rigorous. In fact we need a refined version of Lemma \ref{lem:Vx-asympt} in order to be able to prove Lemma \ref{lemma:Vt-meas} in the cases $\supp\{\mu\}\cap\RR_+=\emptyset$ or $\supp\{\mu\}\cap\RR_-=\emptyset$. We only give full details for the former case as the latter can be addressed by similar methods. 

Let $\supp\{\mu\}\cap\RR_+=\emptyset$ (hence $b_+\equiv+\infty$), then for any $\delta>0$ one has $\mu^\delta(\RR)<1$ and $\lim_{x\to\infty}(G^\delta)'(x)=g_\delta>0$ for some constant $g_\delta$. Therefore Lemma \ref{lem:Vx-asympt} holds in a different form and in particular we claim that
\begin{align}\label{ext2}
\lim_{y\to\infty}\,\,\sup_{0\le s\le T}\big|V^\delta_x(s,y)-g_\delta\big|=0.
\end{align}
If the above limit holds then one can replace \eqref{ext1} in the final part of the proof of Lemma \ref{lemma:Vt-meas} by
\begin{align*}
\sigma_h([x,+\infty))=-\frac{1}{2h}\int^T_{T-h}(g_\delta-V^\delta_x(s,x))ds, 
\end{align*}
and notice that $\big|\sigma_h([x,+\infty))\big|<\eps/2$ for $x$ sufficiently large. Once this is accomplished the rest of the proof of Lemma \ref{lemma:Vt-meas} follows in the same way and one can then repeat the same steps to prove all the remaining properties of $V^\delta_t$.

It remains to prove \eqref{ext2}. As in \eqref{bb} we obtain
\begin{align*}
\EE_x\left[(G^\delta)'(B_{\tau_*})-g_\delta\right]\le V^\delta_x(t,x)-g_\delta\le\EE_x\left[(G^\delta)'(B_{\tau_*}-)-g_\delta\right]. 
\end{align*}
Moreover for any $\eps>0$ there is $x_\eps>0$ such that $|(G^\delta)'(x)-g_\delta|\le \eps$ for $x\in[x_\eps,+\infty)$ and therefore
\begin{align*}
\EE_x\left[\left|(G^\delta)'(B_{\tau_*})-g_\delta\right|\right]\le c \left(\PP_x(\tau_*<T-t)+\PP_x(B_{T-t}\le x_\eps)\right)+\eps. 
\end{align*}
Taking limits as $x\to\infty$ the right-hand side of the expression above goes to $\eps$. Since the latter is arbitrary \eqref{ext2} follows.
\end{proof}

\begin{proof}[Proof of \eqref{limsk}.] For $k\ge 1$ we denote $\mu_-^{(k)}=\mu_-+1/k$. Notice that $\mu_{k}(dx)=\mu_{k+1}(dx)$ for $x\in[-\mu_-,+\infty)$ whereas $\mu_{k+1}(dx)\ge \mu_k(dx)$ for $x\in[-\mu^{(k+1)}_-,-\mu_-)$ since $F'_{\mu_{k+1}}=(k+1)F_{\mu}(-\mu_-)\ge kF_{\mu}(-\mu_-)=F'_{\mu_k}$ on that interval. On the other hand if we denote by $\tau^*_{k+1}$ the optimal stopping time for the problem with value function $V_{k+1}$, we also observe that $L^z_{\tau_{k+1}}=0$, $\PP_{t,x}$-a.s.~for all $z\le -\mu_-^{(k+1)}$ since $b^{(k+1)}_-(t)\le \mu_-^{(k+1)}$ for all $t\in[0,T]$. It then follows for any $(t,x)$
\begin{align*}
\EE_{t,x}\int_\RR L^z_{\tau_{k+1}}\mu_{k+1}(dz)=&\,\EE_{t,x}\int_{[-\mu_-^{(k+1)},+\infty)} L^z_{\tau_{k+1}}\mu_{k+1}(dz)\\[+4pt]
\ge &\,\EE_{t,x}\int_{[-\mu_-^{(k+1)},+\infty)} L^z_{\tau_{k+1}}\mu_{k}(dz)=\EE_{t,x}\int_{\RR} L^z_{\tau_{k+1}}\mu_{k}(dz).
\end{align*}
Therefore we obtain
\begin{align*}
V_{k+1}(t,x)-G_{k+1}(x)=&\,\EE_{t,x}\int_\RR L^z_{\tau_{k+1}}(\nu-\mu_{k+1})(dz)\\[+4pt]
\le &\,\EE_{t,x}\int_\RR L^z_{\tau_{k+1}}(\nu-\mu_{k})(dz)\le V_{k}(t,x)-G_k(x)
\end{align*} 
for all $(t,x)\in[0,T]\times \RR$. For $U_k:=V_k-G_k$, the sequence $(U_k)_{k\ge 0}$ is decreasing. Hence for $\CC_{k}:=\{(t,x)\,:\,U_k(t,x)>0\}$, $k\ge1$ the corresponding continuation sets, one has $\CC_k\supseteq\CC_{k+1}$ for all $k\ge1$. On the other hand it is easy to verify that by construction 
\begin{align*}
\lim_{k\to\infty}\sup_{x\in\RR}|G_k(x)-G(x)|=0
\end{align*}
and therefore also
\begin{align*}
\lim_{k\to\infty}\sup_{(t,x)\in[0,T]\times\RR}|V_k(t,x)-V(t,x)|=0.
\end{align*}
Now arguing exactly as in the proof of \eqref{eq:convCd} and \eqref{eq:convTd} we can demonstrate that $\CC_k\downarrow \CC_T$ and $\sigma^{(k)}_*\downarrow \sigma_*$ $\PP$-a.s.~as required.
\end{proof}

\ackn{This work was funded by EPSRC grant EP/K00557X/1. I am grateful to G.~Peskir for showing me McConnell's work and for many useful discussions related to Skorokhod embedding problems and optimal stopping. I also wish to thank two anonymous referees whose insightful comments helped me to substantially improve the results in this paper}

\end{document}